\documentclass[a4paper,11pt,twoside]{amsart}
\usepackage{amssymb}
\usepackage[OT1]{fontenc}
\usepackage[latin1]{inputenc}
\usepackage{hyperref}
\usepackage[usenames]{color}
\usepackage[curve]{xypic}
\usepackage{enumerate}
\usepackage[active]{srcltx}
\usepackage{ae,aecompl}
\usepackage[mathscr]{euscript}

\numberwithin{equation}{section}

\def\rev{{op}}

\def\pv#1{\ensuremath{{\mathsf{#1}}}}
\def\Om#1#2{\ensuremath{\overline\Omega_{#1}{\pv{#2}}}}
\def\om#1#2{\ensuremath{\Omega_{#1}{\pv{#2}}}}
\def\Omvar#1#2{\ensuremath{\overline\Omega_{#1}{{#2}}}}

\def\Ker#1{\ensuremath{\mathop{\rm Ker}{#1}}}

\def\te#1{\mathsf t_{#1}}
\def\be#1{\mathsf i_{#1}}
\def\op{\lbrack\hskip-1.6pt\lbrack\relax}
\def\cl{\rbrack\hskip-1.6pt\rbrack\relax}

\newcommand\D{\mathcal{D}}
\newcommand\J{\mathcal{J}}
\newcommand\R{\mathcal{R}}
\renewcommand\L{\mathcal{L}}
\renewcommand\H{\mathcal{H}}
\newcommand\K{\mathcal{K}}

\newtheorem{teor}{Theorem}[section]
\newtheorem{prop}[teor]{Proposition}
\newtheorem{lema}[teor]{Lemma}

\newtheorem{cor}[teor]{Corollary}
\newtheorem{remark}[teor]{Remark}

\newtheorem{defi}[teor]{Definition} 

\theoremstyle{definition}

\newcommand\malcev{\mathop{\raise1pt\hbox{\footnotesize$\bigcirc$\kern-8pt\raise1pt
      \hbox{\tiny$m$}\kern1pt}}}

\def\Lo#1{\ensuremath{\mathscr L\pv {#1}}}

\CompileMatrices

\begin{document}

\title[Some operators that preserve
  the locality of a pseudovariety]{Some operators that preserve
  the locality of a pseudovariety of semigroups}

\thanks{Research supported by the Centre for Mathematics of the
University of Coimbra and FCT --- Fundação para a Ciência e a Tecnologia,
 and by the FCT
 project PTDC/MAT/65481/2006, within the framework of the European programmes
 COMPETE and FEDER}

\author{Alfredo Costa}

\author{Ana Escada}
\address{CMUC, Department of Mathematics, University of Coimbra,
  Apartado 3008, EC Santa Cruz,
  3001-501 Coimbra, Portugal.}
\email{\{amgc,apce\}@mat.uc.pt}

\begin{abstract}
  It is shown that if $\pv V$ is a local monoidal pseudovariety of
  semigroups, then $\pv K\malcev \pv V$, $\pv D\malcev \pv V$
  and $\Lo I\malcev \pv V$ are local.
  Other operators of the form $\pv Z\malcev(\_)$ are considered.
  In the process, results about the interplay between operators
  $\pv Z\malcev(\_)$ and $(\_)\ast\pv D_k$ are obtained.
\end{abstract}

\keywords{Local pseudovarieties; permanent pseudovarieties;
  semigroups; monoids; Mal'cev
  products; semidirect products}

\makeatletter
\@namedef{subjclassname@2010}{%
  \textup{2010} Mathematics Subject Classification}
\makeatother
\subjclass[2010]{Primary 20M07, 20M35}

\maketitle

\section{Introduction}

 A pseudovariety of monoids $\pv V$ is \emph{local}
 if the pseudovariety of categories generated by $\pv V$
 coincides
 with the pseudovariety of
 categories whose local monoids belong to $\pv V$.
  Replacing \emph{monoid} by \emph{semigroup}, and
 \emph{category} by
 \emph{semigroupoid}, we get the notion of \emph{local pseudovariety of
 semigroups}. 
 These are important notions due to the key role
 finite categories/semigroupoids often play
 in operations between monoid/semigroup
 pseudovarieties~\cite{Tilson:1987,Rhodes&Steinberg:2009}.
 
  In~\cite{Jones&Szendrei:1992,Jones&Trotter:1995}
  one finds some unary operators
  preserving the locality
  of subpseudovarieties
  of  the pseudovariety $\pv {DS}$ of monoids 
 whose regular~$\D$-classes are semigroups.
  Further examples
  related to the operators studied in this paper appear
  in~\cite{Steinberg:2010LHmV}.
 For a pseudovariety $\pv V$ of monoids/semigroups,
 the pseudovariety of semigroups whose local monoids belong
 to $\pv V$ is denoted by $\Lo V$. 
 In~\cite{Steinberg:2010LHmV}, Steinberg proves that
 when $\pv H$ is a pseudovariety of groups and $\pv V$
 is a local pseudovariety of monoids, then the Mal'cev product
 $\Lo H\malcev  \pv V$ is local
 if $\pv V$ contains the six-element Brandt monoid $B_2^1$ and $\pv H$ is
 Fitting, or if $\pv H$ is extension-closed and
 nontrivial.
  Also in~\cite{Steinberg:2010LHmV}, Steinberg
  posed the following question concerning the
  trivial pseudovariety~$\pv I$:
  is $\Lo I\malcev  \pv V$ local for \emph{every} local
  pseudovariety~$\pv V$ of monoids?
  Since $\pv I$ is a Fitting pseudovariety of groups,
  it suffices to consider pseudovarieties of monoids $\pv V$ such that
  $B_2^1\notin{\pv V}$,~that is, such
  that $\pv V\subseteq \pv {DS}$.
 On the other hand, Almeida had already proved the locality
 of all monoid subpseudovarieties
 of $\pv {DO}$ (the pseudovariety  of monoids 
 whose regular $\D$-classes are orthodox semigroups)
 in the image of the operator $\Lo I\malcev (\_)$~\cite{Almeida:1996c}.
 An example which is not covered by the previous cases, is that of
 the pseudovariety $\Lo I\malcev \pv {CR}$
 (the locality of
 $\pv {CR}$ was proved in~\cite{Jones:1993}). 
 
 We answer Steinberg's question affirmatively.
 First we prove the operator $\pv K\malcev (\_)$
 preserves locality, where $\pv K$ is the pseudovariety of
 semigroups  whose idempotents are left zeros.
 This implies $\pv D\malcev (\_)$
 preserves locality, where $\pv D$ is the dual of
 $\pv K$. As the operator $\Lo I\malcev (\_)$ is
 equal to the infinite alternate
 composition of $\pv K\malcev (\_)$ and $\pv D\malcev (\_)$~\cite{Auinger&Hall&Reilly&Zhang:1995},
 we easily deduce that $\Lo I\malcev (\_)$ preserves locality. 
 We do not take advantage of those cases where Steinberg and Almeida
 had already shown that $\Lo I\malcev (\_)$ preserves locality,
 but when dealing with the restriction
 of $\pv K\malcev (\_)$ to
 subpseudovarieties of $\pv{DS}$,
 we use left basic factorizations of implicit operations,
 a key idea in Almeida's paper~\cite{Almeida:1996c}.

  Results about operators of the form $\pv Z\malcev (\_)$
 for pseudovarieties $\pv Z$ other than
 $\Lo I$, $\pv K$ or $\pv D$, are obtained.
 Many concern the interplay
 between the operators $\pv Z\malcev (\_)$
 and $(\_)\ast\pv D_k$, where $k$ is a positive integer
 and $\pv D_k$ is the pseudovariety of semigroups satisfying
 the identity $yx_1\cdots x_k=x_1\cdots x_k$
  (recall that $\pv  D=\bigcup_{k\geq 1} \pv D_k$).
 Whether $\pv V$ is a monoid or a semigroup pseudovariety,
 the semidirect product $\pv V\ast\pv D_k$ is always a semigroup
 pseudovariety. This leads us to consider
 $\pv Z\malcev (\_)$ as an operator
 on the lattice of semigroup pseudovarieties.
 The translation to an operator on monoid pseudovarieties is easy. 
 As is common in the literature, we often use the same notation
 for a pseudovariety
 of monoids and for its generated pseudovariety of semigroups.
 For example, $\pv {Sl}$ can denote the pseudovariety of
 monoid semilattices or the pseudovariety of
 semigroup semilattices, depending on the context.
 
 As usual, the pseudovarieties of groups
 and of nilpotent semigroups are denoted by
 $\pv G$ and $\pv N$, respectively.
 The set of semigroup pseudovarieties
 \begin{equation*}
 \mathbb V= \{\Lo I,\pv K, \pv D,\pv N,\Lo G,
  \pv K\vee \pv G,\pv D\vee \pv G,\pv N\vee\pv G\}
 \end{equation*}
 plays an important role along this paper, due to the following
 proposition.
  \begin{prop}\label{p:local-operator-comutes}
  Let $\pv Z\in\mathbb V$.
  Then $\Lo {(\pv Z\malcev \pv V)}=\pv Z\malcev \Lo V$.
\end{prop}
The case where $\pv Z=\Lo I$ is Exercise 4.6.58
  in~\cite{Rhodes&Steinberg:2009}.
  The general case $\pv Z\in\mathbb V$ offers no additional
  difficulty,
  but given its importance to our paper, we give an explicit proof
  of Proposition~\ref{p:local-operator-comutes} in the appendix at the
  end of this paper.
  
  Proposition~\ref{p:local-operator-comutes} motivates the search for
  conditions under which the inclusion 
    \begin{equation}
    \label{eq:second-result-introduction-variation}
    \pv Z\malcev(\pv V\ast\pv D_k)\subseteq
          (\pv Z\malcev\pv V)\ast\pv D_k.
  \end{equation}
  holds for all $k\ge 1$,
  since if
  $\Lo {(\pv Z\malcev \pv V)}=\pv Z\malcev \Lo V$ and
  (\ref{eq:second-result-introduction-variation}) hold,
  then the locality of $\pv V$ implies the
  locality of $\pv Z\malcev \pv V$, if $\pv V$ is monoidal (cf.~Lemma~\ref{l:conditions-for-the-locality-of-ZmV}).  
  With this motivation,
  we introduce in Section~\ref{sec:perm-pseud}
  \emph{permanent} pseudovarieties.
  These pseudovarieties have bases of
  pseudoidentities satisfying certain conditions allowing an application of known results
  about pseudoidentities bases
  of Mal'cev
  products~\cite{Pin&Weil:1996a}
  and of semidirect products with
  $\pv {D}_k$~\cite{Almeida&Azevedo&Teixeira:1998}.
  We prove that if $\pv Z$ is permanent then~(\ref{eq:second-result-introduction-variation})
  holds if $\pv Z\malcev\pv V$ is not contained in $\pv {DS}$.
    All members of $\mathbb V$ are permanent.
  The pseudovariety $\pv A$
  of finite aperiodic semigroups
  is also
  permanent.
  
  If $\pv Z\in \{\pv K,\pv D,\Lo I\}$,
  we show that~(\ref{eq:second-result-introduction-variation})
  also holds
  when $\pv {Sl}\subseteq \pv V\subseteq \pv {DS}$. 
    The proof depends strongly on the locality
  of~$\pv {DS}$~\cite{Jones&Trotter:1995}.
  The aforementioned preservation
  of locality by operator $\Lo G\malcev(\_)$,
  proved by Steinberg~\cite{Steinberg:2010LHmV},
  gives another proof of the locality of~$\pv {DS}$,
  since~$\pv {DS}=\Lo G\malcev \pv {Sl}$~\cite{Putcha:1973,Schutzenberger:1976} 
  and $\pv {Sl}$ is local~\cite{Brzozowski&Simon:1973,Eilenberg:1976}. Other examples are discussed in 
  Section~\ref{sec:main-results-appl},
  where the main results and applications appear.
  Among the applications, we deduce, from
  cases where (\ref{eq:second-result-introduction-variation})
  holds, a result (Theorem~\ref{t:varieties-languages}) about
  varieties of formal languages,
  indicating the interest of~(\ref{eq:second-result-introduction-variation})
  is not limited to locality  questions.

The proof of some of the main results in
Section~\ref{sec:main-results-appl} 
is deferred to the three final sections.
One of them, Section~\ref{sec:pseud-pv-dsastpv},
is dedicated to implicit operations on $\pv {DS}\ast\pv D_k$,
and seems to be of independent interest.

\section{Preliminaries}

This paper concerns the theory of pseudovarieties of
semigroups/monoids
(classes of finite semigroups/monoids closed under taking homomorphic images,
 subsemigroups/submonoids and finitary direct products)
which we assume the reader is familiar with. As supporting
references, we indicate books~\cite{Almeida:1994a,Rhodes&Steinberg:2009}.

By an \emph{alphabet} we mean a finite nonempty set.
Let $A$ be an alphabet.
For each pseudovariety of semigroups $\pv V$
we denote by $\Om AV$ the $A$-generated
free pro-$\pv V$ semigroup.
The elements of $\Om AV$ can be interpreted as the \emph{implicit operations}
on $\pv V$, in the sense developed in~\cite[Chapter 3]{Almeida:1994a}
and~\cite[Chapter 3]{Rhodes&Steinberg:2009}.
The subsemigroup of $\Om AV$ generated by $A$, which is dense
in $\Om AV$, is denoted by $\om AV$.  If $\pv V$ contains $\pv N$,
then $\om AV$
is isomorphic to the free semigroup $A^+$,
and every element of $\om AV$
is isolated in $\Om AV$; in this case we identify $\om AV$ with the
semigroup $A^+$.
The elements of $\Om AV$ can be viewed as
generalizations of words on the alphabet $A$, and so
sometimes they are also called \emph{pseudowords} (of $\Om AV$).

Since the cardinal of an alphabet determines $\Om AV$ and $\om AV$, we may
write $\Om {|A|}V$ and $\om {|A|}V$
instead of $\Om AV$ and $\om {A}V$, respectively.
Let $\pv S$ be the pseudovariety of all finite semigroups.
The semigroup $\Om 2S$ will play an important role in this paper.
The elements of the generating set of $\Om 2S$ are the implicit operations
$x_1$ and $x_2$, where $x_i$ is the binary implicit operation
``projection on the $i$-th component''.

An ordered pair $(u,v)$ of implicit operations of $\Om AV$ can be
interpreted as a \emph{pseudoidentity} over the alphabet $A$,
and as such is denoted by the formal equality $u=v$.
We will also use
the standard notation $S\models u=v$ and $\pv V\models u=v$ to
indicate that the profinite semigroup $S$ and
the pseudovariety $\pv V$ satisfy the pseudoidentity $u=v$, and we write
$\pv V=\op\Sigma\cl$ to indicate that $\pv V$ is defined by the basis
of pseudoidentities~$\Sigma$.
Again, the reader looking for details may consult
\cite[Chapter 3]{Almeida:1994a} or \cite[Chapter 3]{Rhodes&Steinberg:2009}.

The \emph{content} of a pseudoword $u\in\Om AV$, with $\pv
{Sl}\subseteq \pv V$, is the image of
$u$ by the canonical homomorphism $c:\Om AV\to\Om A{Sl}$.
It amounts to the set of letters on which $u$ depends~\cite[Section 8.1]{Almeida:1994a}.
One defines~$c(1)=\emptyset$.

For a semigroup $S$ with operation $\star$,
its \emph{dual}, denoted $S^{\rev}$, is
the semigroup
with the same underlying set
and operation $\star_{\rev}$ such that
$s\star_{\rev}t=t\star s$ for all $s,t\in S$.
If $S$ is a topological semigroup, then $S^{\rev}$ is a topological
semigroup for the same topology.
The \emph{dual of a pseudovariety of semigroups $\pv V$} is
the pseudovariety 
$\pv V^\rev=\{S^\rev\mid S\in\pv V\}$.
The pseudovariety $\pv V$ is \emph{self-dual} if~$\pv V=\pv V^\rev$.

Given an alphabet $A$,
the semigroup $(\Om AS)^\rev$ is profinite, and so
there is a unique
continuous homomorphism $\chi_A$ from $\Om AS$ to $(\Om AS)^\rev$
such that $\chi_A(a)=a$ for every $a\in A$.
Applying a forgetful functor, we may view
$\chi_A$ as a continuous self-mapping of $\Om AS$,
and thus we may consider the composition
$\chi_A\circ\chi_A$.
Clearly, one has $\chi_A\circ\chi_A(A^+)=A^+$.
Since $A^+$ is dense in $\Om AS$, we have
$\chi_A\circ\chi_A=id_{\Om AS}$. Therefore
$\chi_A$ is a continuous isomorphism from $\Om AS$ onto $(\Om
AS)^\rev$~(cf.~\cite[Exercise 5.2.1]{Almeida:1994a}).

\begin{lema}\label{l:properties-of-reverse-homo}
  Let $u,v\in \Om AS$.
  For a finite semigroup $S$, one has $S\models u=v$,
if and only if $S^\rev\models\chi_A(u)=\chi_A(v)$.
\end{lema}

\begin{proof}
  In the case where $u,v\in A^+$, the lemma is justified
  by a simple induction argument.
  For the general case, note that,
  since $S$ is finite, $\chi_A$ is continuous and $A^+$ is dense in
  $\Om AS$, there are $x,y\in A^+$ such that $S\models u=x$,
  $S\models v=y$, $S^\rev\models \chi_A(u)=\chi_A(x)$
  and $S^\rev\models \chi_A(v)=\chi_A(y)$.
  Hence, $S\models u=v$ if and only if $S\models x=y$,
  if and only if $S^\rev\models\chi_A(x)=\chi_A(y)$ (by the already
  proved case), if and only if $S^\rev\models\chi_A(u)=\chi_A(v)$.  
\end{proof}

Given a semigroup $S$ which is not a monoid,
one denotes by $S^1$ the monoid obtained by $S$ by adjoining an
identity; if the semigroup $S$ is a monoid, then one defines $S^1=S$.
Occasionally it will be convenient
to consider the profinite monoid $(\Om AV)^1$, obtained from
$\Om AV$ by adjoining the empty word~$1$ as an identity and an
isolated point. If $\pv S$ is a profinite semigroup and $\pv V$ is a
semigroup pseudovariety, then $S\models 1=1$ and
$\pv V\models 1=1$.

For a pseudovariety of monoids $\pv V$, the
pseudovariety of semigroups generated by the elements of $\pv V$
viewed as semigroups is denoted by $\pv V_{\pv S}$.
It is well known that $S\in\pv V_{\pv S}$ if and only if $S^1\in\pv
V$~(\cite[Section 7.1]{Almeida:1994a}).
A pseudovariety of semigroups is \emph{monoidal} if it is of the form
$\pv V_{\pv S}$ for some pseudovariety of monoids $\pv V$.

Let $\pv W$ be a pseudovariety of semigroups.
As is usual in the literature,
regardless of whether $\pv V$  is a pseudovariety of semigroups
or of monoids, we shall denote by $\pv W\malcev \pv V$ the Mal'cev
product of $\pv W$ with $\pv V$
and by $\pv W\ast\pv V$ the semidirect product
of $\pv W$ by $\pv V$.

Recall that, if $\pv V$ is a
pseudovariety of semigroups, then $\pv W\malcev \pv V$ is the
pseudovariety of semigroups generated by finite semigroups $S$ for which
there is a semigroup homomorphism $\varphi$ from~$S$ into some element~$T$ of
$\pv V$ such that $\varphi^{-1}(e)\in\pv W$ for every idempotent
$e$ in~$T$. Replacing \emph{semigroup} by \emph{monoid}
in the previous sentence,
we obtain the characterization of $\pv W\malcev \pv V$ when $\pv V$ is
a pseudovariety of monoids.
We remark the following difference between the Mal'cev product of
two semigroup pseudovarieties and the
Mal'cev product of a semigroup pseudovariety with a monoid
pseudovariety:
in the former case, $\pv W\malcev \pv V$ contains $\pv W$,
while in the latter case, that may not occur.
In both cases, $\pv W\malcev \pv V$ contains $\pv V$.
The following proposition
seems folklore.

\begin{prop}\label{p:malcev-sgp-vs-monoid}
  Let $\pv W$ be a pseudovariety of semigroups, and let $\pv V$ be a
  pseudovariety of monoids.
  If\/ $\pv V$ contains $\pv {Sl}$
  then $(\pv W\malcev \pv V)_{\pv S}=\pv W\malcev \pv V_{\pv S}$.
\end{prop}

\begin{proof}
  Clearly, $\pv W\malcev \pv V\subseteq\pv W\malcev \pv V_{\pv S}$,
  thus $(\pv W\malcev \pv V)_{\pv S}\subseteq\pv W\malcev \pv V_{\pv
    S}$.

  Conversely, let $S$ be a finite semigroup for which there is a
  homomorphism $\varphi:S\to T$ such that $T\in\pv V_{\pv S}$
  and $\varphi^{-1}(e)\in \pv W$ for every idempotent $e$ of $S$.
  For a semigroup $R$, denote by $R^I$ the monoid obtained from $R$ by adjoining an identity
  $I$ not in $R$ (regardless of whether $R$ is a monoid or not).
  Extend $\varphi:S\to T$ to a
  monoid homomorphism $\bar\varphi:S^I\to T^I$.
  Suppose that $T$ is not a monoid.
  Then $T^I\in\pv V$, since $T\in\pv V_{\pv S}$,
  and so, as $\bar\varphi^{-1}(I)=\{I\}$, we conclude that
  $S^I\in \pv W\malcev \pv V$, thus $S\in (\pv W\malcev \pv V)_{\pv S}$.
  Suppose now that $T$ is a monoid with identity $1_T$.
  Consider the monoid semilattice $U_1=\{0,1\}$, with the usual multiplication.
  Then $T\times U_1\in \pv V$.
  Note that $T\times\{0\}\cup \{(1_T,1)\}$ is a submonoid of $T\times U_1$
  isomorphic to $T^I$. Hence, $T^I\in\pv V$. Again,
  from the extension of $\varphi$ to a monoid homomorphism
  $\bar\varphi:S^I\to T^I$ we deduce that
  $S^I\in \pv W\malcev \pv V$ and so $S\in (\pv W\malcev \pv V)_{\pv S}$.
  This concludes the proof that
  $\pv W\malcev \pv V_{\pv  S}\subseteq (\pv W\malcev \pv V)_{\pv S}$.
\end{proof}

Proposition~\ref{p:malcev-sgp-vs-monoid} fails if $\pv V$ does not contain
the monoid semilattices. For example, viewing the pseudovariety
$\pv G$ of groups as a semigroup pseudovariety, we have
$\pv K\malcev \pv G=\pv K\vee\pv G$, while interpreting
$\pv G$ as a monoid pseudovariety we have $\pv K\malcev \pv G=\pv G$
(the deduction of these equalities are
easy exercises, cf.~\cite[Exercise 4.6.9]{Rhodes&Steinberg:2009}).

\begin{prop}
  \label{p:reverse-malcev}
  We have
  $(\pv W\malcev \pv V)^\rev=\pv W^\rev\malcev \pv V^\rev$
  for every pair $\pv V$, $\pv W$ of pseudovarieties of semigroups.
\end{prop}

\begin{proof}
  Note that if $\pv V$ is generated by a class $\mathcal K$ of
  semigroups then $\pv V^\rev$ is generated by the class
  $\{S^\rev\mid S\in \mathcal K\}$.
  So, let $S$ be a generating
  element of $\pv W\malcev \pv V$
  for which there is
  a homomorphism $\varphi:S\to T$ such that $T$ belongs
  to $\pv V$
  and $\varphi^{-1}(e)\in\pv W$ for every idempotent $e$ from $T$.
  We may consider
  the homomorphism $\varphi_\rev:S^\rev\to T^\rev$
  which as a set-theoretic mapping coincides with
  $\varphi$.
  Note that $\varphi_\rev^{-1}(e)=\varphi^{-1}(e)^\rev$. Therefore,
  $S^\rev\in \pv W^\rev\malcev \pv V^\rev$, thus proving
  $(\pv W\malcev \pv V)^\rev\subseteq\pv W^\rev\malcev \pv V^\rev$.
  Since the pseudovarieties are arbitrary,
  we also have $\pv W^\rev\malcev \pv V^\rev=
  ((\pv W^\rev\malcev \pv V^\rev)^\rev)^\rev\subseteq (\pv W\malcev \pv V)^\rev$.
\end{proof}

\begin{prop}[{\cite{Straubing:1985}}]\label{p:reverse-VastD}
  If $\pv V$ is a
  pseudovariety of semigroups containing some nontrivial monoid,
  then $(\pv V\ast\pv D_k)^\rev=\pv V^\rev\ast\pv D_k$ for every
  $k\ge 1$, and thus $(\pv V\ast\pv D)^\rev=\pv V^\rev\ast\pv D$.
\end{prop}

We remark that
in the statement of Proposition~\ref{p:reverse-VastD}
we choose  the semigroup varietal version
stated in~\cite[Exercise 10.6.9]{Almeida:1994a}, instead of the
original monoid version from~\cite{Straubing:1985}.

\section{Permanent pseudovarieties}\label{sec:perm-pseud}

\begin{defi}{\rm
Let $u=v$ be a pseudoidentity
between
implicit operations $u,v$ of $\Om 2S$
such that $u^2=u$, $u(u,v)=u$ and $v(u,v)=v$.
If additionally we have $v=uv$ (respectively $v=vu$)
then we say that $u=v$ is a
\emph{left-permanent}
(respectively \emph{right-permanent})
pseudoidentity.

A pseudovariety of semigroups
is \emph{left-permanent} (respectively~\emph{right-perma\-nent})
if it has a basis of left-permanent
(respectively \emph{right-permanent})
pseudoidentities.

A pseudovariety which is the intersection of
a family of left-permanent and right-permanent pseudovarieties is
called a \emph{permanent} pseudovariety.
Alternatively, a
permanent pseudovariety is
a pseudovariety having a basis consisting of
permanent pseudoidentities,
where by \emph{permanent pseudoidentity}
we mean a pseudoidentity which is either
left-permanent or right-permanent.
}
\end{defi}

  \begin{prop}\label{p:reverse-left}
    A pseudovariety of semigroups $\pv Z$ is left-permanent if
    and only if the pseudovariety
    $\pv Z^\rev$ is right-permanent.
  \end{prop}

  \begin{proof}
    Since a left-permanent (respectively right-permanent)
    pseudovariety is the intersection of
    pseudovarieties defined by only one
    left-permanent (respectively right-permanent)
    pseudoidentity,
    it suffices to consider the case where
    $\pv Z=\op u=v\cl$ for some left-permanent
    pseudoidentity $u=v$.

    Then, by Lemma~\ref{l:properties-of-reverse-homo},
    we have $\pv Z^\rev=\op \chi_2(u)=\chi_2(v)\cl$.
    Now,
    \begin{equation*}
    \chi_2(v)=\chi_2(uv)=\chi_2(u)\cdot_{\rev}\chi_2(v)=
    \chi_2(v)\cdot\chi_2(u).      
    \end{equation*}
    It is also clear that $\chi_2(u)$ is an idempotent.
    
    Finally,
    let $\varphi$ be the unique continuous endomorphism
    of $\Om 2S$ such that
    $\varphi(x_1)=u$
    and $\varphi(x_2)=v$,
    and
    consider the unique continuous endomorphism
    $\psi$ of $\Om 2S$ such that $\psi(x_i)=\chi_2(\varphi(x_i))$,
    for $i\in\{1,2\}$.
    By induction on the length of $u$,
    it is routine to check that
    $\psi(u)=\chi_2\circ\varphi\circ\chi_2(u)$ for every $u\in \om 2S$
    (recall that the composition $\chi_2\circ\varphi\circ\chi_2$ is
    well defined via the application of a forgetful functor).
    Hence, since $\om 2S$ is dense in $\Om 2S$ and $\psi$, $\varphi$
    and $\chi_2$ are continuous mappings, it follows that
    the function $\psi$ equals the composite
    $\chi_2\circ\varphi\circ\chi_2$.
    Then
    \begin{equation*}
      \chi_2(u)(\chi_2(u),\chi_2(v))=
      \psi(\chi_2(u))=
      \chi_2\circ\varphi\circ\chi_2\circ\chi_2(u).
    \end{equation*}
    Since $\chi_2\circ\chi_2$ is the identity and
    $\varphi(u)=u(u,v)=u$,
    we conclude that
    $\chi_2(u)(\chi_2(u),\chi_2(v))=\chi_2(u)$.
    Similarly,
    $\chi_2(v)(\chi_2(u),\chi_2(v))=\chi_2(v)$.
    Therefore,
    $\chi_2(u)=\chi_2(v)$ is a right-permanent pseudoidentity
    and $\pv Z^\rev$ is a right-permanent pseudovariety.
  \end{proof}

We now give some examples. 
The five pseudoidentities
\begin{equation*}
x_1^\omega=x_1^\omega x_2,\quad
x_1^\omega=x_1^\omega x_2^\omega,\quad x_1^\omega=x_1^\omega x_2x_1^\omega,
\quad x_1^\omega=(x_1^\omega x_2x_1^\omega)^\omega,
\quad x_2^\omega=x_2^{\omega+1},
\end{equation*}
are left-permanent. The corresponding left-permanent pseudovarieties defined by
each of them are $\pv K$, $\pv K\vee \pv G$,
$\Lo I$, $\Lo G$ and $\pv A$, respectively.
The pseudovarieties
$\pv D$, $\pv D\vee \pv G$,
$\Lo I$, $\Lo G$
and $\pv A$ are right-permanent pseudovarieties,
a fact easy to verify directly, but that also follows
immediately from
Proposition~\ref{p:reverse-left}
because $\pv D=\pv K^\rev$,
$\pv D\vee \pv G=(\pv K\vee \pv G)^\rev$
and $\Lo I$, $\Lo G$, $\pv A$ are self-dual.
Therefore, $\pv N=\pv K\cap\pv D$ and
$\pv N\vee\pv G=(\pv K\vee \pv G)\cap (\pv D\vee \pv G)$
are permanent pseudovarieties.

In all the preceding examples, the permanent pseudoidentities
are formed by $\omega$-words. That is not the case of the following example.
For a prime~$p$, the pseudovariety $\pv G_p$ of $p$-groups
is defined by the pseudoidentity $1=x^{p^{\omega}}$,
where $x^{p^{\omega}}=\lim_{n\to\infty} x^{p^{n!}}$~\cite{Rhodes&Steinberg:2009},
so $\Lo {G}_p$ is defined by the left-permanent pseudoidentity
$x_1^\omega=(x_1^\omega x_2 x_1^\omega)^{p^\omega}$.

\begin{lema}\label{l:K-is-contained-left-permanent}
  If $\pv Z$ is a left-permanent semigroup pseudovariety then $\pv K\subseteq
  \pv Z$. Dually, 
  if $\pv Z$ is a right-permanent semigroup pseudovariety then $\pv D\subseteq
  \pv Z$.
\end{lema}

\begin{proof}
  Let $u=v$ be a left-permanent pseudoidentity.
  Then $u=u^2$ and $v=uv$.
  The equality $u=u^2$ implies $u\notin \om 2S$,
  thus $\pv K\models u=uv$. That is,
  $\pv K\subseteq \op u=v\cl$, since $v=uv$.
  If $\pv Z$ is left-permanent, then
  it is the intersection of pseudovarieties
  of the form $\op u=v\cl$, with $u=v$ a left-permanent
  pseudoidentity,
  and so $\pv K\subseteq \pv Z$.
  If $\pv Z$ is right-permanent,
  we reduce to the left-permanent case
  by application of Proposition~\ref{p:reverse-left}.
\end{proof}

  \begin{defi}\label{def:collection-of-substitutions}
    {\rm
    Let $\Sigma$ be a set of pseudoidentities between elements
    of $\Om 2S$.
    For each $(u=v)\in\Sigma$, let
    $\mathscr {F}_{(u=v)}$ be a set of continuous homomorphisms
    of the form
    $\varphi:\Om 2S\to \Om {n(\varphi)}S$, where $n(\varphi)$ runs
    over the set of positive integers.
    Denote by  $\mathscr {F}$ the family $(\mathscr
    {F}_{(u=v)})_{(u=v)\in  \Sigma}$. Consider the set of
    pseudoidentities
    \begin{equation*}
      \Sigma_{\mathscr {F}}= \{\varphi(u)=\varphi(v)\mid (u=v)\in\Sigma,\,\varphi\in \mathscr {F}_{(u=v)}\}.
\end{equation*}
Let $\pv V$ be the pseudovariety $\op \Sigma_{\mathscr {F}}\cl$.
We say that
$\mathscr {F}$ is a
 \emph{$\Sigma$-collection of substitutions defining $\pv V$}.
}
  \end{defi}

We use the language introduced in
Definition~\ref{def:collection-of-substitutions} to
make the following statement of the Pin-Weil basis theorem
for Mal'cev products, for the special case where the first
pseudovariety in the Mal'cev product is defined by
a basis of pseudoidentities between elements of $\Om 2S$.
  
  \begin{teor}[cf.~{\cite[Theorem 4.1]{Pin&Weil:1996a}}]
    \label{t:base-for-malcev-product}
    Let $\pv Z$ and $\pv V$ be two pseudovarieties of semigroups.
    Suppose that $\pv Z$
    has a basis $\Sigma$ of pseudoidentities
    between elements of $\Om 2S$.
    Consider the set $\mathscr {H}$ of continuous homomorphisms
    of the form
    $\varphi:\Om 2S\to \Om {n(\varphi)}S$, where $n(\varphi)$ runs
    over the set of positive integers, such that
    $\pv V\models \varphi(x_1)=\varphi(x_2)=\varphi(x_2)^2$.
    For each $(u=v)\in\Sigma$, take $\mathscr {F}_{(u=v)}=\mathscr {H}$.
    Then $\pv Z\malcev\pv V=\op \Sigma_{\mathscr {F}}\cl$.
  \end{teor}

 The ``only if'' part of the following lemma is crucial
  to obtain many of our main results, because of
  its application in
  the proof of Proposition~\ref{p:malcev-stability-preserved-by-D-2}.
  
\begin{lema}\label{l:malcev-idempotent-preliminar}
  Let $\pv Z$ be a pseudovariety of semigroups
  with a basis $\Sigma$ consisting of permanent pseudoidentities.
  Let $\pv V$ be a pseudovariety of semigroups.
  Then there is a $\Sigma$-collection of substitutions
  defining~$\pv V$
  if and only if $\pv Z\malcev\pv V=\pv V$.
\end{lema}

  \begin{proof}
      Let $\pv V$ be a pseudovariety of semigroups
  for which there is a $\Sigma$-collection $\mathscr {F}$ of substitutions
  defining~$\pv V$.
    Let $(u=v)\in \Sigma$ and $\varphi\in\mathscr {F}_{(u=v)}$.
    As $u=u^2$, we have $\varphi(u)=\varphi(u)^2$.
    Hence $\pv V\models \varphi(u)=\varphi(v)=\varphi(v)^2$.
    Since $\pv Z\models u=v$,
    from Theorem~\ref{t:base-for-malcev-product}
    we obtain
    $\pv Z\malcev \pv V\models
    u(\varphi(u),\varphi(v))=v(\varphi(u),\varphi(v))$.
    Note that, for each $w\in\{u,v\}$, we have
    $w(\varphi(u),\varphi(v))=\varphi(w(u,v))=\varphi(w)$.
    Hence $\pv Z\malcev \pv V\models \varphi(u)=\varphi(v)$.
    Therefore $\pv Z\malcev \pv V\models\Sigma_{\mathscr {F}}$,
    that is $\pv Z\malcev \pv V\subseteq \pv V$.
    The reverse
   inclusion is trivial.

     Conversely, if~$\pv Z\malcev \pv V=\pv V$, then it follows
     immediately from Theorem~\ref{t:base-for-malcev-product}
     that there is a $\Sigma$-collection of substitutions
     defining~$\pv V$.
  \end{proof}

  \begin{cor}\label{c:malcev-idempotent}
    Let
    $\pv Z$ be a permanent pseudovariety of semigroups.
    For every pseudovariety of semigroups $\pv V$
    we have $\pv Z\malcev(\pv Z\malcev \pv V)=\pv Z\malcev\pv V$.
    In particular, $\pv Z=\pv Z\malcev \pv Z$.
  \end{cor}

  \begin{proof}
         By Theorem~\ref{t:base-for-malcev-product},
         $\pv Z\malcev \pv V$
         is defined by a $\Sigma$-collection of substitutions,
         so the result follows immediately
         from Lemma~\ref{l:malcev-idempotent-preliminar}.
  \end{proof}

\section{Main results and applications}\label{sec:main-results-appl}

\subsection{Interplay between operators $\pv Z\malcev(\_)$ and
  $(\_)\ast\pv {D}_k$}

The proof of the following theorem is deferred to
Section~\ref{sec:proof-theorem-reft:z}. Recall that $B_2$ denotes the
five-element Brandt semigroup, and that $B_2\notin\pv V$ if and only if
$\pv V\subseteq\pv {DS}$,
for every
semigroup pseudovariety $\pv V$~\cite[Lemma 2.2.3]{Rhodes&Steinberg:2009}.
  
  \begin{teor}
    \label{t:z-malc-ast-d}
        Let $\pv Z$ be a permanent pseudovariety of
        semigroups.
        Let $\pv V$ be a pseudovariety of semigroups
        such that $\pv Z\malcev \pv V=\pv V$.
        Suppose moreover that
        $B_2\in\pv V$, or that
        $\pv V$ is local and contains some nontrivial monoid.
        Then        
         \begin{equation*}
           \pv Z\malcev(\pv V\ast\pv D_k)=\pv V\ast\pv D_k
         \end{equation*}
      for every~$k\geq 1$.
  \end{teor}

  As a corollary of Theorem~\ref{t:z-malc-ast-d}, we get the following result.
  
  \begin{teor}
    \label{t:more-z-malc-ast-d}
        Let $\pv Z$ be a permanent pseudovariety of
        semigroups.
        Let $\pv V$ be a pseudovariety of semigroups
        such that $B_2\in\pv Z\malcev\pv V$. Then
        \begin{equation*}
          \pv Z\malcev(\pv V\ast\pv D_k)\subseteq
          (\pv Z\malcev\pv V)\ast\pv D_k
       \end{equation*}
       for every~$k\geq 1$.
  \end{teor}

  \begin{proof}
    By Corollary~\ref{c:malcev-idempotent},
    we have $\pv Z\malcev(\pv Z\malcev\pv V)=\pv Z\malcev\pv V$.
    Hence, from Theorem~\ref{t:z-malc-ast-d}
    we obtain the equality
    \begin{equation*}
    \pv Z\malcev((\pv Z\malcev\pv V)\ast\pv D_k) =(\pv Z\malcev\pv
    V)\ast\pv D_k.
    \end{equation*}
    On the other hand,
    clearly
    $\pv Z\malcev(\pv V\ast\pv D_k)\subseteq \pv Z\malcev((\pv
    Z\malcev\pv V)\ast\pv D_k)$.
  \end{proof}

  Recall that all pseudovarieties of $\mathbb V$ are permanent.
  Let $\pv Z\in\mathbb V$.
  Then $\pv {DS}=\pv Z\malcev \pv {DS}$. This can be easily
  deduced directly, using the fact that $B_2\notin\pv V$ if and only if
  $\pv V\subseteq \pv {DS}$, for every semigroup pseudovariety $\pv
  V$.
  Alternatively, one can deduce it
  from equality 
  $\pv {DS}=\Lo G\malcev \pv
  {Sl}$~\cite{Putcha:1973,Schutzenberger:1976}
  and inclusion $\pv Z\subseteq \Lo G$.
  Hence, if $\pv V\subseteq \pv {DS}$
  then $B_2\notin\pv Z\malcev\pv V$ and so
  Theorem~\ref{t:more-z-malc-ast-d} does not apply.
    
 \begin{teor}\label{t:more-K-malc-ast-d}
  Let $\pv Z\in \{\pv K,\pv D,\Lo I\}$.
  If $\pv V$
  is a pseudovariety
  of semigroups containing $\pv {Sl}$ then
         \begin{equation}
          \label{eq:special-cases-KDLI-1}
          \pv Z\malcev(\pv V\ast\pv D_k)
          \subseteq (\pv Z\malcev\pv V)\ast\pv D_k
        \end{equation}
for all $k\geq 1$.  
\end{teor}

 Note that if $B_2\in\pv V$ then~(\ref{eq:special-cases-KDLI-1}) holds
 by Theorem~\ref{t:more-z-malc-ast-d}.
 We shall prove Theorem~\ref{t:more-K-malc-ast-d}
 in Section~\ref{sec:proof-theor-reft:m}, after
 some preparations in Section~\ref{sec:pseud-pv-dsastpv}
 for dealing with the case
 in which $B_2\notin\pv V$.

 Theorem~\ref{t:more-K-malc-ast-d}
 has the following immediate corollary, which improves
 Theorem~\ref{t:z-malc-ast-d} in some cases.
 
 \begin{cor}\label{c:improves-stability}
   Let $\pv Z\in \{\pv K,\pv D,\Lo I\}$.
  If $\pv V$
  is a pseudovariety
  of semigroups containing $\pv {Sl}$ and such that
  $\pv Z\malcev\pv V=\pv V$, then
           \begin{equation*}
          \pv Z\malcev(\pv V\ast\pv D_k)=\pv V\ast\pv D_k
        \end{equation*}
for all $k\geq 1$.  
 \end{cor}

\subsection{Preservation of locality}

By Tilson's Delay Theorem~\cite{Tilson:1987},
a monoid pseudovariety $\pv V$ is local
if and only if $\Lo V=\pv V\ast\pv D$.
In contrast, there are non-local semigroup pseudovarieties
that are solutions of equation~$\Lo X=\pv X\ast\pv D$.
That is the case of the nontrivial pseudovarieties of groups
viewed as semigroup pseudovarieties (cf.~\cite[Theorem 10.6.14]{Almeida:1994a} and the
discussion in~\cite[page 14]{Almeida&Weil:1996});
on the other hand,
nontrivial pseudovarieties of groups are local as monoid pseudovarieties~\cite{Straubing:1985}.
But if $\pv V$ is a monoid pseudovariety not contained in $\pv G$,
then $\pv V$ is local if and only if
$\pv V_{\pv S}$ is local, from which it follows that
for monoidal pseudovarieties of semigroups not contained
in $\pv G$, there is equivalence between being local and being a solution
of $\Lo X=\pv X\ast\pv D$~\cite{Almeida&Weil:1996}.

\begin{lema}\label{l:conditions-for-the-locality-of-ZmV}
  Let $\pv Z$ and $\pv V$ be pseudovarieties
  of semigroups for which we have
  $\Lo {(\pv Z\malcev \pv V)}=\pv Z\malcev \Lo V$
  and $\pv Z\malcev(\pv V\ast\pv D)\subseteq (\pv Z\malcev\pv
  V)\ast\pv D$.
  If $\pv V$ is
  a solution of the equation $\Lo X=\pv X\ast\pv D$, then so is
  $\pv Z\malcev \pv V$.
  In particular, if~$\pv V$ is monoidal and local,
  then $\pv Z\malcev \pv V$ is local.
\end{lema}

\begin{proof}
  By hypothesis,
  $\Lo {(\pv Z\malcev \pv V)}
  =\pv Z\malcev \Lo V=\pv Z\malcev(\pv V\ast\pv D)
  \subseteq (\pv Z\malcev\pv V)\ast\pv D$,
  and so $\pv Z\malcev \pv V$ is a solution of the equation
  $\Lo X=\pv X\ast\pv D$ (recall that $\pv W\ast\pv D\subseteq \Lo W$
for every pseudovariety of semigroups~$\pv W$~\cite[Proposition
10.6.13]{Almeida:1994a}).

Suppose now that $\pv V$ is monoidal and local. By the remarks
before the statement of the lemma, $\pv V\nsubseteq \pv G$
and $\pv V$ is a solution of the equation
$\Lo X=\pv X\ast\pv D$. As we proved in the previous paragraph,
$\pv Z\malcev\pv V$ is also a solution
of the equation, and so, since $\pv Z\malcev\pv V$ is monoidal
(by Proposition~\ref{p:malcev-sgp-vs-monoid}), we conclude
that $\pv Z\malcev\pv V$ is local.
\end{proof}

\begin{teor}\label{t:locality-of-ZmV}
  Let $\pv Z\in\mathbb V$
  and let $\pv V$ be a monoidal pseudovariety of semigroups such that
  $B_2\in \pv V$.
  If $\pv V$ is local then
  $\pv Z\malcev \pv V$ is local.
\end{teor}

\begin{proof}
  Thanks to Proposition~\ref{p:local-operator-comutes}
  and Theorem~\ref{t:more-z-malc-ast-d}, we are in the conditions of
  Lemma~\ref{l:conditions-for-the-locality-of-ZmV}, and so one gets
  immediately the theorem.
\end{proof}
  
\begin{teor}\label{t:locality-of-kmv}
  Let $\pv V$ be a monoidal pseudovariety of semigroups.
  If $\pv V$ is local then the pseudovarieties of semigroups
  $\pv K\malcev \pv V$,
  $\pv D\malcev \pv V$ and $\Lo I\malcev \pv V$
  are local.
\end{teor}

\begin{proof}
  Since $\pv V$ is local, we do not have $\pv V\subseteq \pv G$.
  Therefore, $\pv V$ contains $\pv {Sl}$, since it is monoidal.
  The theorem now follows immediately from
  Proposition~\ref{p:local-operator-comutes},
  Theorem~\ref{t:more-K-malc-ast-d}
  and Lemma~\ref{l:conditions-for-the-locality-of-ZmV}.
\end{proof}

Next, we translate Theorem~\ref{t:locality-of-kmv} to the context of
monoid pseudovarieties.

\begin{teor}\label{t:locality-of-kmv-monoidal-version}
  Let $\pv V$ be a pseudovariety of monoids.
  If $\pv V$ is local then the pseudovarieties of monoids $\pv K\malcev \pv V$,
  $\pv D\malcev \pv V$ and $\Lo I\malcev \pv V$
  are local.
\end{teor}

\begin{proof}
  Let $\pv Z\in\{\pv K, \pv D,\Lo I\}$.
  If $\pv V$ consists only of groups, then $\pv Z\malcev \pv V=\pv V$.
  If $\pv V$ contains a monoid which is not a group, then
  $\pv V_{\pv S}$ contains $\pv {Sl}$. Since
  $\pv V_{\pv S}$ is local, so is $\pv Z\malcev \pv V_{\pv S}$,
  by Theorem~\ref{t:locality-of-kmv}. But
  $\pv Z\malcev \pv V_{\pv S}=(\pv Z\malcev \pv V)_{\pv S}$,
  by Proposition~\ref{p:malcev-sgp-vs-monoid}.
  Therefore, $\pv Z\malcev \pv V$ is local.
\end{proof}

Theorems~\ref{t:locality-of-ZmV}
and~\ref{t:locality-of-kmv}
unify several known results concerning locality. For example, it is well known that
the pseudovarieties $\pv R$ of
$\R$-trivial semigroups, and $\pv {DA}$ of semigroups whose regular
$\D$-classes are aperiodic semigroups, are
local~\cite{Stiffler:1973,Almeida:1996c}.
The equalities $\pv R=\pv K\malcev \pv {Sl}$
and $\pv {DA}=\Lo I\malcev \pv {Sl}$ are
also well known, and were used in Almeida's proof of the locality of
$\pv R$ and $\pv {DA}$~\cite{Almeida:1996c}.
Their locality may be seen as an
application of Theorem~\ref{t:locality-of-kmv},
since $\pv {Sl}$ is local~\cite{Brzozowski&Simon:1973,Eilenberg:1976}.

More generally, let $(\pv R_m)_{m\ge 1}$ and $(\pv L_m)_{m\ge 1}$ be
the families of pseudovarieties recursively defined by
$\pv R_1=\pv L_1=\pv {Sl}$ and
  $\pv R_{m+1}=\pv K \malcev \pv L_{m}$,
  $\pv L_{m+1}=\pv D \malcev \pv R_{m}$ for $m\ge 1$.
  These two families,  each one with union equal to $\pv {DA}$, have received some
 attention~\cite{Trotter&Weil:1997,Kufleitner&Weil:2009,Kufleitner&Weil:2010,Kufleitner&Weil:2012}\footnote{In~\cite{Kufleitner&Weil:2009,Kufleitner&Weil:2010,Kufleitner&Weil:2012}.
   one has $\pv R_1=\pv L_1=\pv J$, where $\pv J$  is the
   pseudovariety of $\J$-trivial semigroups. But the
   remaining pseudovarieties are the same, since
   $\pv K\malcev \pv {Sl}=\pv K\malcev \pv J$ and
   $\pv D\malcev \pv {Sl}=\pv D\malcev \pv J$. We prefer to define $\pv R_1=\pv L_1=\pv {Sl}$ because $\pv J$ is not local~\cite{Knast:1983b}.}.
 As $\pv {Sl}$ is local, Theorem~\ref{t:locality-of-kmv} provides the first proof, as far as we know,
 that all pseudovarieties in this pair of families 
 are local. 

 As an example of an application
 of Theorem~\ref{t:locality-of-ZmV},
 we give an alternative proof
 of a result of Steinberg,
which states in particular that ${\pv G}\ast {\pv A}$ is local.
If $\pv H$ is a pseudovariety of groups, then
$\bar{\pv H}$ denotes the pseudovariety of semigroups whose maximal subgroups
belong to $\pv H$.
One has $\bar{\pv H}=\Lo {\bar{\pv H}}$
(cf.~\cite[Chapter V, Proposition 10.6]{Eilenberg:1976}),
and so $\bar{\pv H}$ is local,
because
$\bar{\pv H}$ is monoidal
and every pseudovariety of
the form $\Lo V$
is a solution of the equation
$\Lo X= \pv X\ast\pv D$.

\begin{teor}[{\cite{Steinberg:2004a}}]\label{t:application-to-GA}
  If $\pv H$ is a pseudovariety of groups
  closed under semidirect product, then
  $\pv G\ast\bar{\pv H}$ is local.
\end{teor}

\begin{proof}
  The pseudovariety $\bar {\pv H}$ is also closed under semidirect
  product~\cite[Corollary 4.1.47]{Rhodes&Steinberg:2009}.
  Moreover, it is closed under adjoining identities. Therefore,
  we are in the conditions of Theorem 4.11.4
  from~\cite{Rhodes&Steinberg:2009}, and so
  $\pv G\ast\bar{\pv H}=(\pv {K}\vee\pv {G})\malcev \bar{\pv H}$.
  As $\bar{\pv H}$ is local
   and contains $B_2$,
  and $\pv {K}\vee\pv {G}\in\mathbb V$, the result follows from
  Theorem~\ref{t:locality-of-ZmV}.
\end{proof}

Note that,
for every pseudovariety $\pv H$ of groups,
Proposition~\ref{p:local-operator-comutes} can be used to
obtain
$(\pv {K}\vee\pv {G})\malcev \bar{\pv H}=
\Lo [(\pv {K}\vee\pv {G})\malcev \bar{\pv H}]$,
and so in the proof of Theorem~\ref{t:application-to-GA},
one can apply Proposition~\ref{p:local-operator-comutes}
instead of Theorem~\ref{t:locality-of-ZmV}.

  The operator $\pv N\malcev(\_)$ does not preserve locality of
  subpseudovarieties of~$\pv {DS}$.
  Indeed, the pseudovariety 
  $\pv {J}$ of $\J$-trivial semigroups
  is equal to $\pv {N}\malcev \pv {Sl}$~\cite{Pin:1986;bk}, and it is not
  local~\cite{Knast:1983b}, while
  $\pv {Sl}$ is local.
 
  We do not know whether in
  the cases $\pv Z\in \{\pv K\vee \pv G,\pv D\vee \pv G,
  \pv N\vee \pv G\}$ 
  the operator $\pv Z\malcev(\_)$ preserves locality when restricted
  to the interval $[\pv {Sl},\pv {DS}]$.
    Since $\pv {DG}=(\pv N\vee \pv G)\malcev\pv{Sl}$ (cf.~proof of
  Proposition 5.12 in~\cite{Pin&Weil:1996a}),
  research on this subject could lead to a new proof that $\pv {DG}$
  (the pseudovariety of semigroups whose regular
  $\D$-classes are groups) is local, a difficult result  proved by
  Ka\v dourek~\cite{Kadourek:2008}.

\subsection{An application to varieties of formal
  languages}

 Eilenberg's theorem relating pseudovarieties of
 semigroups (or monoids) with
 varieties of formal languages gave a strong motivating framework to
 study the former. The following theorem summarizes a series of
 results which are good examples of the adequacy of that framework.

 \begin{teor}[\cite{Pin:1980b,Pin:1986;bk,Pin&Straubing&Therien:1988,Straubing:1979a},
   see also survey~\cite{Pin:1997}\footnote{In \cite{Pin:1997} one
     only finds the  characterization
     of $\ast$-varieties of languages closed under concatenation
     product.
     Nevertheless, with the appropriated translation,
   that characterization holds for $+$-varieties of languages,
   thanks to~\cite[Theorem 4.1]{Chaubard&Pin&Straubing:2006},
   a more general result.}]\label{t:characterizations-of-languages}
 Given a $+$-variety $\mathcal V$ of languages,
 let  $\pv V$ be the pseudovariety of semigroups
 generated by the syntactic semigroups of elements of $\mathcal V$.
 Then, we have:
 \begin{enumerate}
 \item $\mathcal V$ is closed under concatenation product
   if and only if $\pv V=\pv A\malcev\pv V$;
 \item $\mathcal V$ is closed under unambiguous product if and only if
    $\pv V=\Lo I\malcev\pv V$;
 \item $\mathcal V$ is closed under left deterministic product
   if and only if $\pv V=\pv K\malcev\pv V$,
 and dually, $\mathcal V$ is closed under right deterministic product
 if and only if $\pv V=\pv D\malcev\pv V$.
 \end{enumerate}   
 \end{teor}

 We denote by
 $(\_)\ast \mathcal D_k$ the operator
 in the lattice of $+$-varieties of languages
 corresponding to the operator
 $(\_)\ast \pv D_k$, and by~$\mathcal Sl$ the $+$-variety of languages
 recognized by semigroups of $\pv {Sl}$.

\begin{teor}\label{t:varieties-languages}
  Let $\mathcal V$ be a variety of $+$-languages.
  \begin{enumerate}
  \item If $\mathcal V$ is closed under concatenation product, then
    so is $\mathcal V\ast\mathcal D_k$.
    \item If $\mathcal V$  contains $\mathcal Sl$ and is
      closed under unambiguous product
      (respectively, left deterministic product, right deterministic
      product), then $\mathcal V\ast\mathcal D_k$
      is also
      closed under unambiguous product
      (respectively, left deterministic product, right deterministic
      product).
  \end{enumerate}
\end{teor}

\begin{proof}
  Thanks to Theorem~\ref{t:characterizations-of-languages}, this is an
  immediate application of Theorem~\ref{t:z-malc-ast-d} (recall
  that $\pv A$ is permanent and $B_2\in \pv A$) and
  Corollary~\ref{c:improves-stability}.
\end{proof}

\section{Proof of Theorem~\ref{t:z-malc-ast-d}}\label{sec:proof-theorem-reft:z}

  We assume familiarity with the fundamentals of profinite semigroupoids,
  namely the equational theory of semigroupoid pseudovarieties.
  See~\cite{Jones:1996,Almeida&Weil:1996,Rhodes&Steinberg:2009}\footnote{For
    delicate questions related with infinite-vertex
  semigroupoids, see also~\cite{Almeida&Costa:2009}.
  However, all semigroupoids in this paper are finite-vertex.}.
  Recall in particular that for a pseudovariety of semigroups $\pv V$,
  the pseudovariety of semigroupoids generated by $\pv V$,
  the \emph{global of $\pv V$}, is denoted by $g\pv V$.
  The pseudovariety of semigroupoids whose local semigroups belong to
  $\pv V$, the \emph{local of $\pv V$}, is denoted by~$\ell\pv V$.
  Hence $\pv V$ is local if and only if $g\pv V=\ell\pv V$.
  The free profinite semigroupoid,
  over the pseudovariety $\pv {Sd}$ of
  all finite semigroupoids,
  and generated by a finite graph $A$,
  is denoted by $\Om A{Sd}$.

  \begin{defi}
    \label{def:global-Z-based}
        {\rm Let $\pv Z$ be a pseudovariety of semigroups with a
    basis $\Sigma$ of left-permanent pseudoidentities.
    Let $\pv W$ be a pseudovariety of semigroupoids such that
    $\pv W$ has a basis
    \begin{equation*}
      \Gamma=\bigcup_{(u=v)\in\Sigma}\Gamma_{(u=v)}
    \end{equation*}
    in which an element $\mathfrak p$ of $\Gamma_{(u=v)}$, with
    $(u=v)\in\Sigma$, is
    a path pseudoidentity of the form
    $(u_L(r_1,r_2)=v_L(r_1,r_2);A_{\mathfrak p})$, where $A_{\mathfrak p}$ is a
    finite graph, $L$ is the local semigroup
    at a vertex $c$ of
    $\Om {A_{\mathfrak p}}{Sd}$, and $r_1$ and $r_2$ are
    elements of $L$.
    We say that $\pv W$ is~\emph{left $\pv Z$-based}.
    }
  \end{defi}

  In~\cite{Almeida&Azevedo&Teixeira:1998},
  given an alphabet $X$,
  to each pseudoword $w\in\Om XS$
  it is associated a finite graph denoted by $A_w$.
  In this paper we shall not recall
  the somewhat technical definition of $A_w$, but we
  will next highlight some of its features which will be used in the
  proof of
  Proposition~\ref{p:-when-global-is-Z-based-left-version}.
  The precise definitions and
  results can be found in~\cite[Section 5]{Almeida&Azevedo&Teixeira:1998}.
  
  The edges of $A_w$ are elements of $X$,  
  and every element of $X^+$ that labels
  a path $p$ in $A_w$ can be interpreted
  as being the path $p$.
  The pseudoword $w$ can also be interpreted uniquely as
  a profinite path of $A_w$,
  as follows from the next lemma.
  \begin{lema}[{\cite[Lemma 5.7]{Almeida&Azevedo&Teixeira:1998}}]\label{l:interpretation-of-w}
    If $(w_n)_{n\ge 1}$ is a sequence of elements of $X^+$
    converging to $w$, then
    there is $N\ge 1$ such that
    $(w_n)_{n\ge N}$ is a sequence of elements of $\Om {A_w}{Sd}$
    converging to a profinite path $L$ of $A_w$. The profinite path
    $L$ depends only on $w$,
    not on the choice of the sequence  $(w_n)_{n\ge 1}$.    
  \end{lema}

  The limit $L$ in Lemma~\ref{l:interpretation-of-w}
  is the interpretation of $w$ in $\Om {A_w}{Sd}$.
  
  It follows also from Lemma~\ref{l:interpretation-of-w}
  that every factor $u$ of $w$
  can be interpreted as a profinite path of $A_w$,
  in the manner which we next explain.
  Let $x,y\in(\Om XS)^1$ be such that $w=xuy$.
  Consider a sequence $(x_n,u_n,y_n)_n$ of elements of
  $X^\ast\times X^+\times X^\ast$ converging to
  $(x,u,y)$. Then, for large enough~$n$,
  the word $x_nu_ny_n$ can be interpreted as an element of 
  $\Om {A_w}{Sd}$ by Lemma~\ref{l:interpretation-of-w}.
  In particular, $u_n$ can be interpreted as an
  element of $\Om {A_w}{Sd}$. Therefore, in this manner, an accumulation point
  of $(u_n)_n$ in $\Om {A_w}{Sd}$ can be seen as an interpretation of
  $u$ in $\Om {A_w}{Sd}$. From a careful reading of the
  paragraph
  preceding~\cite[Lemma 5.7]{Almeida&Azevedo&Teixeira:1998},
  where Lemma~\ref{l:interpretation-of-w} is proved,
  one concludes that the interpretation of $u$ is unique, but we shall
  not need to use this more precise information.

  The following property will be used without reference.
  
  \begin{prop}[cf.~{\cite[Proposition 5.6]{Almeida&Azevedo&Teixeira:1998}}]
    For an alphabet $X$, let $u,v\in\Om XS$.
    If $B_2\models u=v$ then $A_u=A_v$.
  \end{prop}

  The next result explains the importance of the graphs of the
  form~$A_w$.

  \begin{teor}[{\cite[Theorem 5.9]{Almeida&Azevedo&Teixeira:1998}}]\label{t:basis-V-to-gV}
    Let $\pv V$ be a pseudovariety of semigroups containing $B_2$.
    If $\pv V=\op u_i=v_i\mid i\in I\cl $
    then $g\pv V=\op (u_i=v_i;A_{u_i})\mid i\in I\cl$.
  \end{teor}

  Theorem~\ref{t:basis-V-to-gV} is crucial for proving the following
  proposition. 

  \begin{prop}
    \label{p:-when-global-is-Z-based-left-version}
        Let $\pv Z$ be a left-permanent pseudovariety of semigroups,
    and let $\pv V$ be a
    pseudovariety of semigroups
    such that $\pv V=\pv Z\malcev \pv V$.
    If\/ $\pv V$ is local
    or if $B_2\in\pv V$, then $g\pv V$ is left $\pv Z$-based.
  \end{prop}

  \begin{proof}
    Let $\Sigma$ be a basis of~$\pv Z$
    comprised by left-permanent pseudoidentities.
    Let $\mathscr {H}$ be the set of continuous homomorphisms
    and $\mathscr {F}$ the $\Sigma$-collection of substitutions
    as described in Theorem~\ref{t:base-for-malcev-product}.
    Then
    $\pv V=\op\Sigma_{\mathscr {F}} \cl$, since
    $\pv V=\pv Z\malcev \pv V$.
    
    The pseudovariety $\ell\pv V$ has a basis
    formed by  path pseudoidentities
    of the form $(\varphi(u)=\varphi(v); X_{\varphi})$,
    where $(u=v)$ runs over $\Sigma$ and $\varphi:\Om 2S\to \Om {X_{\varphi}}S $
    runs over~$\mathscr H$.
    Therefore, $\ell\pv V$ is left $\pv Z$-based.
    In particular, if $\pv V$ is local then
    $g\pv V$ is left~$\pv Z$-based.

    Suppose now that $B_2\in\pv V$.
    The set of path pseudoidentities of the form
    \begin{equation}\label{eq:-when-global-is-Z-based-left-version-1}
    \Bigl(\varphi(u)=\varphi(v);\,A_{\varphi(v)}\Bigr),\quad
    \end{equation}    
    where
    $(u=v)$ runs over $\Sigma$ and $\varphi$ runs over $\mathscr H$,
    defines a basis for $g\pv V$, by Theorem~\ref{t:basis-V-to-gV}.

    Fix $(u=v)\in\Sigma$
    and $\varphi\in \mathscr H$.
    Suppose first that $|c(v)|=2$.
    Then $\varphi(x_1)$ and $\varphi(x_2)$
    are factors of $\varphi(v)$
    and so they can be interpreted as profinite paths in the graph
    $A_{\varphi(v)}$.
    We denote these paths by $r_1$ and $r_2$ respectively.
    We want to show that $r_1$ and $r_2$ are loops rooted at the same
    vertex.
    
    By the definition of $\mathscr H$,
    we have $\pv V\models \varphi(x_1)=\varphi(x_2)$.
    Since $\pv V=\pv Z\malcev \pv V$, we have
    $\pv Z\subseteq \pv V$, thus
    $\pv K\models \varphi(x_1)=\varphi(x_2)$
    by Lemma~\ref{l:K-is-contained-left-permanent}.
    In particular, $\varphi(x_1)$ and $\varphi(x_2)$
    start with same letter, and so $r_1$ and $r_2$ have the same
    origin.
    It will be convenient to use the notation $\alpha(e)$ and
    $\omega(e)$
    for the origin and the terminus of an edge $e$, respectively.
    Hence,
    \begin{equation}
      \label{eq:equal-origin-1}
    \alpha(r_1)=\alpha(r_2).
    \end{equation}

    We next show that $\omega(r_1)=\omega(r_2)$.
    We treat the case $x_1\in c(u)$ (the case $x_2\in c(u)$ is
    analogous).
    As $v=uv$ and
    $c(v)=\{x_1,x_2\}$, the word $x_1x_2$ is a factor of~$v$, thus
    $\varphi(x_1)\varphi(x_2)$ is a factor of
    $\varphi(v)$.
    Therefore,
    \begin{equation}
      \label{eq:equal-origin-2}
      \omega(r_1)=\alpha(r_2).
    \end{equation}
    Hence, by~(\ref{eq:equal-origin-1})
    and~(\ref{eq:equal-origin-2}), $r_1$ is a loop rooted at vertex
    $c=\alpha(r_2)$.
    If $|c(u)|=2$ or $x_2$ is not the last letter of $v$, then
    $x_2x_1$ is a factor of $v$, thus
    $\varphi(x_2)\varphi(x_1)$ is a factor of
    $\varphi(v)$
    and $\omega(r_2)=\alpha(r_1)=c$.
    The other case to consider occurs if
    $x_2\notin c(u)$ and $x_2$ is the last letter of $v$.
    Then $\omega(r_1)=\omega(\varphi(u))$
    and $\omega(r_2)=\omega(\varphi(v))$.
    But $\omega(\varphi(u))=\omega(\varphi(v))$, by the definition of
    path pseudoidentity, and so $\omega(r_2)=\omega(r_1)=c$.
    In both cases, we conclude that $r_2$ is a loop rooted at $c$.
    
    Denote by $L$ the local semigroup of $A_{\varphi(v)}$ at $c$.
    We claim that if $w\in \{u,v\}$
    then the interpretation
    of the pseudoword
    $\varphi(w)$
    as a profinite path of $A_{\varphi(v)}$
    is precisely the profinite path
    $w_L(r_1,r_2)$. To prove the claim, let $(w_n)_n$, $(r_{1,n})_n$
    and $(r_{2,n})_n$ be sequences of finite words
    converging respectively to the pseudowords $w$,
    $\varphi(x_1)$ and $\varphi(x_2)$,
    where $r_{i,n}$ as a path of
    $A_{\varphi(v)}$ converges to $r_i$, for $i\in\{1,2\}$.
    Since the sequence of finite words $(w_n(r_{1,n},r_{2,n}))_n$
    converges to
    $w_{\Om {X_{\varphi}}S}(\varphi(x_1),\varphi(x_2))=\varphi(w)$,
    by Lemma~\ref{l:interpretation-of-w} it suffices to prove that
    the sequence of finite paths $(w_n(r_{1,n},r_{2,n}))_n$ converges to
    $w_L(r_1,r_2)$ in $\Om {A_{\varphi(v)}}{Sd}$.
    We may assume that $r_{1,n}$ and $r_{2,n}$ are elements of the
    local semigroup $L$ for all $n$,
    and so $w_n(r_{1,n},r_{2,n})$ as a path
    of $A_{\varphi(v)}$ equals $(w_n)_L(r_{1,n},r_{2,n})$.
    Let $\psi$ be a continuous homomorphism from
    $\Om {A_{\varphi(v)}}{Sd}$ onto a finite semigroupoid $S$.
    Then $\psi$ induces, by restriction, a continuous semigroup homomorphism
    $\psi_c$ from $L$ into the local semigroup  of $S$ at the vertex
    $\psi(c)$, denoted $S_{\psi(c)}$.
    Then, for large enough $n$,
    \begin{align*}
      \psi_c((w_n)_L(r_{1,n},r_{2,n}))&=(w_n)_{S_{\psi(c)}}(\psi_c(r_{1,n}),\psi_c(r_{2,n}))\\
      &=(w_n)_{S_{\psi(c)}}(\psi_c(r_{1}),\psi_c(r_{2}))\\
      &=w_{S_{\psi(c)}}(\psi_c(r_{1}),\psi_c(r_{2}))\\
      &=\psi_c(w_L(r_{1},r_{2})).
    \end{align*}
    Hence $((w_n)_L(r_{1,n},r_{2,n}))_n$
    converges to $w_L(r_{1},r_{2})$, proving the claim.
    Therefore, when $|c(v)|=2$,
    the pseudoidentity~(\ref{eq:-when-global-is-Z-based-left-version-1})
    is actually the pseudoidentity    
    \begin{equation}\label{eq:-when-global-is-Z-based-left-version-2}
    \Bigl(u_L(r_1,r_2)=
    v_L(r_1,r_2);\,A_{\varphi(v)}\Bigr).
    \end{equation}
    
    Let us now consider the case $|c(v)|=1$.
    Then $c(u)=c(v)$ because $v=uv$. 
    If $c(v)=\{x_1\}$, then, as $u=u^2$ and $v=v(u,v)$,
    we have $u=x_1^\omega=v$.
    Suppose that $c(v)=\{x_2\}$
    and let $r$ be the interpretation of
    $\varphi(x_2)$ as a profinite path of $A_{\varphi(v)}$.
    Since $x_2^2$ is factor of $v$,
    we know that $r$ is a loop rooted at a vertex $c$.
    For
    each $w\in\{u,v\}$,
    we have $w_{\Om {X_{\varphi}}S}(\varphi(x_1),\varphi(x_2))=
    w_{\Om {X_{\varphi}}S}(\varphi(x_2),\varphi(x_2))$.    
    Then, with the same kind of arguments used
    in the case $|c(v)|=2$,
    we conclude that 
    the interpretation
    of $w_{\Om {X_{\varphi}}S}(\varphi(x_1),\varphi(x_2))$
    as a profinite path of $A_{\varphi(v)}$
    is the profinite path $w_L(r,r)$.
    
    Putting cases $|c(v)|=1$ and $|c(v)|=2$
    together,
    we conclude that $g\pv V$ has a basis of pseudoidentities
    of the form~(\ref{eq:-when-global-is-Z-based-left-version-2}),
    which proves that $g\pv V$ is left $\pv Z$-based.    
  \end{proof}

  For the proof of the next proposition we need to add some more
  notation and definitions.
  
  For a graph $A$, we denote by $V(A)$ the set of its vertices, and
  by $E(A)$ the set of its edges.

  The following definitions are taken
  from~\cite[Section 5.2]{Almeida:1994a}. They will be used also in
  Section~\ref{sec:pseud-pv-dsastpv}.
  If $u$ is an element of $\Om AS$, then there is a unique word
  $\te n(u)$ of length at most $n$ such that $\pv D_n\models u=\te n(u)$.
  If $u$ is a word of $A^+$ of length at most $n$, then
  $u=\te n(u)$, otherwise $\te n(u)$ is the unique word $w$
  of length $n$ such that $u\in(\Om AS)^1\cdot w$. Dually, one may
  consider the unique word $\be n(u)$ such that
  $\pv K_n\models u=\be n(u)$, where $\pv K_n$ is the dual of $\pv  D_n$. 
  These definitions are extended to the empty word by
  letting $\be n(1)=\te n(1)=1$.
  
   \begin{prop}\label{p:malcev-stability-preserved-by-D-2}
     Let $\pv Z$ be a left-permanent pseudovariety of semigroups,
     and let $\pv V$ be a pseudovariety of semigroups
     such that $g\pv V$ is left $\pv Z$-based.
     Then  $\pv Z\malcev(\pv V\ast \pv D_n)=\pv V\ast \pv D_n$,
     for every $n\geq 1$.
   \end{prop}       
 
  \begin{proof}
    For $\pv W=g\pv V$, we retain the notation from Definition~\ref{def:global-Z-based}.
    By the application of the Almeida-Weil
    basis theorem\footnote{The Almeida-Weil basis theorem describes a
      basis of pseudoidentities for semidirect products
      $\pv V\ast\pv W$ of semigroup pseudovarieties.
      The proof has a gap found
      by Rhodes and Steinberg, but it works when
      $g\pv V$ has a basis of pseudoidentities with bounded number of
      vertices or when $\pv W$ is locally finite (cf.~\cite[Theorem 3.7.15]{Rhodes&Steinberg:2009}).
      Since $\pv {D}_n$ is locally finite, the result holds in the
      cases in which we are interested.}  to semidirect products
    with~${\pv D}_n$,
    which is explained in the proof of Theorem 4.1
    in~\cite{Almeida&Azevedo&Teixeira:1998},    
    the pseudovariety 
    $\pv V\ast \pv D_n$ has a basis
    \begin{equation*}
      \Upsilon=\bigcup\,\Bigl\{\Upsilon_{\mathfrak p}
      \mid (u=v)\in\Sigma,\,{\mathfrak p}\in\Gamma_{(u=v)}\Bigr\}
    \end{equation*}
    where each $\Upsilon_{\mathfrak p}$
    has the following property:
     an element $\mathfrak e$ of $\Upsilon_{\mathfrak p}$
      is a pseudoidentity of the form 
          \begin{equation}
      \label{eq:special-basis-for-VDn}
    \pi_c\cdot \delta(u_L(r_1,r_2))=\pi_c\cdot\delta(v_L(r_1,r_2)),      
    \end{equation}
    where $\delta:\Om {A_{\mathfrak p}}{Sd}\to \Om {B_{\mathfrak e}}S$
    is a continuous semigroupoid
    homomorphism and $(\pi_d)_{d\in V(A_{\mathfrak p})}$
    is a family of words of $B_{\mathfrak e}^+$ with length at most $n$,
    such that, among other restrictions, one has
          \begin{equation}
      \label{eq:compatibility-condition}
    \pv D_n\models \pi_{\alpha s}\delta(s)=
    \pi_{\omega s},\quad\text{for every $s\in E(A_{\mathfrak p})$}.
    \end{equation}
     We shall denote by $\mathscr {D}_{\mathfrak p}$
     the set of all continuous homomorphisms $\delta$ arising in this way, for a
     fixed $\mathfrak p$.
     
    Let us fix $(u=v)\in \Sigma$, $\mathfrak p\in \Gamma_{(u=v)}$
    and $\mathfrak e\in \Upsilon_{\mathfrak p}$ with the above
    definitions and notations.
    
    Since $u_L(r_1,r_2)=u_L(r_1,r_2)^2$, the profinite path
    $u_L(r_1,r_2)$ is not finite, thus $u_L(r_1,r_2)=ts_1\cdots s_n$
    for some $t\in E(\Om {A_{\mathfrak p}}{Sd})$ and
    $s_1,\ldots,s_n\in E(A_{\mathfrak p})$. 
    Note that $\omega s_n=c$.
    Applying (\ref{eq:compatibility-condition})
    consecutively, we then get
    \begin{equation*}
    \pv D_n\models \pi_{\alpha s_1}\delta(s_1s_2s_3\cdots s_n)=
    \pi_{\alpha s_2}\delta(s_2s_3\cdots s_n)=
    \cdots =
    \pi_{\omega s_n}=\pi_c,      
    \end{equation*}
    which means that
    $\te n(\delta(u_L(r_1,r_2)))=\te n(\delta(s_1\cdots s_n))=\pi_{c}$.
    Therefore, there is a pseudoword $w$ such that
    $\delta(u_L(r_1,r_2))=w\pi_c$.
    Hence, if a semigroup $S$ satisfies~(\ref{eq:special-basis-for-VDn}),
    it also satisfies $\delta(u_L(r_1,r_2))^2=
    \delta(u_L(r_1,r_2)v_L(r_1,r_2))$.
    Since $u=u^2$ and $v=uv$,
    we conclude that $S$ satisfies~(\ref{eq:special-basis-for-VDn})
    if and only if it satisfies
        \begin{equation}
      \label{eq:special-basis-for-VDn-modified}
    \delta(u_L(r_1,r_2))=
    \delta(v_L(r_1,r_2)).
    \end{equation}
    Let
    $\psi_{\delta}$ be the continuous homomorphism $\Om 2S\to \Om
    {B_{\mathfrak e}}S$ such that
    $\psi_{\delta}(x_i)=\delta(r_i)$
    for each $i\in\{1,2\}$.
    Then~(\ref{eq:special-basis-for-VDn-modified}),
        which is equivalent to~(\ref{eq:special-basis-for-VDn}),
    can be rewritten as
    \begin{equation*}
          \psi_\delta(u)=\psi_{\delta}(v).
    \end{equation*}
        Therefore, denoting by
    $\mathscr {F}_{(u=v)}$ the set
        $\{\psi_{\delta}
    \mid \delta\in\mathscr {D}_{\mathfrak  p},\,\mathfrak p\in\Gamma_{(u=v)}\} $,
    we proved that the family
    $(\mathscr {F}_{(u=v)})_{(u=v)\in \Sigma}$
    is a $\Sigma$-collection of substitutions
    defining $\pv V\ast \pv D_n$.
    The result now follows from Lemma~\ref{l:malcev-idempotent-preliminar}.
  \end{proof}

 \begin{proof}[Conclusion of the proof of
    Theorem~\ref{t:z-malc-ast-d}]
            Let $\pv Z$ be a permanent pseudovariety of
        semigroups,
        and let $\pv V$ be a pseudovariety of semigroups
        such that $\pv Z\malcev \pv V=\pv V$.
        Suppose also that
        $B_2\in\pv V$,
        or that $\pv V$ is local and contains some nontrivial monoid.
    If $\pv Z$ is left-permanent, it follows immediately from
    Propositions~\ref{p:-when-global-is-Z-based-left-version}
    and~\ref{p:malcev-stability-preserved-by-D-2}
    that the equality
     \begin{equation}
          \label{eq:z-malc-ast-d-2}
          \pv Z\malcev(\pv V\ast\pv D_k)=\pv V\ast\pv D_k
    \end{equation}
    holds.

    If $\pv Z$ is right-permanent,
    then $\pv Z^\rev$ is
    left-permanent by Proposition~\ref{p:reverse-left}.
    Also, if $\pv V$ is local, then $\pv V^\rev$ is local (an easy
    consequence
    of a remark made in the paragraph before Lemma 19.3
    from~\cite{Tilson:1987}), and
    $B_2\in \pv V$ implies $B_2\in\pv V^\rev$.
    Hence, from the already proved case, together
    with Propositions~\ref{p:reverse-VastD}
    and~\ref{p:reverse-malcev},
    we  obtain
    \begin{equation*}
      (\pv Z\malcev(\pv V\ast\pv D_k))^\rev=
      \pv Z^\rev\malcev(\pv V^\rev\ast\pv D_k)
      =\pv V^\rev\ast\pv D_k=(\pv V\ast\pv D_k)^\rev,
    \end{equation*}
    proving equality (\ref{eq:z-malc-ast-d-2}) for the case of
    right-permanent pseudovarieties.

    Finally, suppose that $\pv Z$ is
    a permanent pseudovariety of semigroups.
    Then $\pv Z=\bigcap_{i\in I}\pv Z_i$
    for some family $(\pv Z_i)_{i\in I}$
    consisting of left-permanent or right-permanent pseudovarieties of
    semigroups.
    The Mal'cev product of pseudovarieties
    is right-distributive over intersections
    of pseudovarieties~\cite[Corollary 3.2]{Pin&Weil:1996a},
    and therefore we have $\pv Z\malcev (\pv V\ast\pv D_k)=
    \bigcap_{i\in I}\pv Z_i\malcev (\pv V\ast\pv D_k)$.
    This concludes the proof, since we have
    shown that
    $\pv Z_i\malcev (\pv V\ast\pv D_k)=\pv V\ast\pv D_k$.
  \end{proof}

\section{The pseudovarieties $\pv {DS}\ast\pv D_k$}\label{sec:pseud-pv-dsastpv}

In this section we obtain some properties of
implicit operations on $\pv {DS}\ast\pv D_k$
similar to fundamental properties of
implicit operations on $\pv {DS}$ presented in~\cite[Chapter 8]{Almeida:1994a}.
The former properties are obtained by reduction to the latter via
general properties of semidirect products of the form
$\pv V\ast\pv D_k$
and the use of iterated left basic factorizations of implicit operations.

\subsection{Semidirect products of the form $\pv V\ast\pv D_k$ and
  the mapping $\Phi_k$}\label{sec:mapping-phi_k}

Let $k$ be a positive integer.
Consider the mapping $A^+\to (A^{k+1})^\ast$
that maps to $1$
each word of $A^+$ with length less than $k+1$,
and maps each word $u$ of $A^+$ with length at least $k+1$
to the word of $(A^{k+1})^+$
formed by the consecutive factors of length $k+1$ of $u$.
This mapping has a unique continuous extension to
a mapping
$\Phi_k:\Om AS\to (\Om {A^{k+1}}S)^1$~\cite[Lemma 10.6.11]{Almeida:1994a}.
Note that, for $u\in\Om AS$, one has $\Phi_k(u)=1$ if and only if
$u$ has length less than $k+1$.
The mapping $\Phi_k$
has the following property, which we use frequently without
explicit reference: for every $u,v\in \Om AS$
one has
\begin{equation*}
  \Phi_k(uv)=\Phi_k(u\,\be k(v))\cdot \Phi_k(v)=\Phi_k(u)\cdot\Phi_k(\te k(u)v).
\end{equation*}
One may informally think
that, for $u\in\Om AS$, the pseudoword $\Phi_k(u)$ is the result of ``reading''
the consecutive factors of length $k+1$ of $u$.

The mapping $\Phi_k$ appears frequently in the study
of semidirect products of the form $\pv V\ast\pv D_k$.
It is in this context that it is introduced
in~\cite[Section 10.6]{Almeida:1994a}.

\begin{teor}[{\cite[Theorem 10.6.12]{Almeida:1994a}}]\label{t:basic-phi-k}
  Let $\pv V$ be a pseudovariety of semigroups containing some
  nontrivial monoid.
  Let $u,v\in\Om AS$.
  Then $\pv V\ast\pv D_k\models u=v$ if and only if
  $\be k(u)=\be k(v)$,
  $\te k(u)=\te k(v)$ and $\pv V\models \Phi_k(u)=\Phi_k(v)$.  
\end{teor}

For the following corollary of Theorem~\ref{t:basic-phi-k},
we introduce some notation.
For an alphabet $A$ and a pseudovariety of semigroups $\pv V$, the
image of an element $u$ of $\Om AS$ by the canonical homomorphism
$\Om AS\to \Om AV$ will be denoted by $[u]_{\pv V}$.
Recall that $\pv V\models u=v$ if and only if $[u]_{\pv V}=[v]_{\pv V}$.
More generally, if $\theta$ is a binary relation on $\Om AV$ and
$[u]_{\pv V}\mathrel {\theta}[v]_{\pv V}$, then
we may use the notation $\pv V\models u\mathrel {\theta}v$.

\begin{cor}\label{c:basic-phi-k}
  Let $\pv V$ be a pseudovariety of semigroups containing some
  nontrivial monoid. Let $\K$ be one of the Green's relations
  $\J$, $\R$ or $\L$.
  For every $u,v\in\Om AS$, the following properties hold:
  \begin{enumerate}
  \item if $\pv V\ast\pv D_k\models u\leq_{\K}v$ then
  $\pv V\models \Phi_k(u)\leq_{\K}\Phi_k(v)$;\label{item:basic-phi-k-green-1}
  \item if $[u]_{\pv V\ast\pv D_k}$ is regular then $[\Phi_k(u)]_{\pv  V}$
    is regular\label{item:basic-phi-k-green-2}.
  \end{enumerate}
\end{cor}

\begin{proof}
  (\ref{item:basic-phi-k-green-1}):
  We write the proof for $\K=\R$ only, since the other cases are similar.
  If $\pv V\ast\pv D_k\models u\leq_{\R}v$,
  then $\pv V\ast\pv D_k\models u=vx$ for some $x\in(\Om AS)^1$.
  Hence $\pv V\models \Phi_k(u)=\Phi_k(vx)$
  by Theorem~\ref{t:basic-phi-k}.
  As $\Phi_k(vx)=\Phi_k(v)\cdot \Phi_k(\te k(v)x)$, this establishes
  $\pv V\models \Phi_k(u)\leq_{\R}\Phi_k(v)$.
  
  (\ref{item:basic-phi-k-green-2}): If $[u]_{\pv V\ast\pv D_k}$ is
  regular then $\pv V\ast\pv D_k\models u=uxu$, for some $x\in\Om AS$.
  Then
  $\pv V\models\Phi_k(u)=\Phi_k(uxu)$ by Theorem~\ref{t:basic-phi-k}.
  Since $\Phi_k(uxu)=\Phi_k(u)\cdot\Phi_k(\te k(u)x\,\be k(u))\cdot\Phi_k(u)$,
  this establishes (\ref{item:basic-phi-k-green-2}).
\end{proof}

As another consequence of Theorem~\ref{t:basic-phi-k},
we have $c(\Phi_k(u))=c(\Phi_k(v))$
if $\pv {Sl}\ast\pv D_k\models u=v$.
Sometimes we use the
notation $c_{k+1}(u)$ for $c(\Phi_k(u))$. Note that $c_{k+1}(u)$ is
just the set of factors of length $k+1$ of $u$.

\subsection{Left basic factorizations}\label{sec:left-basic-fact}

Left basic factorizations, whose definition we next recall, are
explained in detail and extensively used in~\cite[Section
3]{Trotter&Weil:1997}, which we consider as our supporting reference
for this particular subject.
This tool had already proved to be quite useful
in~\cite{Almeida:1996c,Almeida&Weil:1996b}.

Let $\pv V$ be a pseudovariety of semigroups
containing
$\pv {Sl}$.
Let $u$ be an element of $\Om AV$. A \emph{left basic factorization}
of $u$ is a factorization $u=xay$ such that
$x,y\in (\Om AV)^1$, $a\in A$, $a\notin c(x)$ and $c(u)=c(xa)$.
Every element of $\Om AV$ has a left basic factorization.

Suppose furthermore that $\pv V=\pv K\malcev \pv V$.
Then the left basic factorization of $u\in\Om AV$ is unique: if
$u=xay$ and $u=zbt$ are left basic factorizations of $u$
(where $x,y,z,t\in (\Om AV)^1$ and $a,b\in A$)
then $x=z$, $a=b$ and $y=t$. This uniqueness property
is stated in~\cite[Proposition 3.1]{Trotter&Weil:1997}. The proof
is done in~\cite[Proposition 2.3.1]{Almeida&Weil:1996b}
for a special class of pseudovarieties of the
form $\pv V=\pv K\malcev\pv V$, but the proof holds for all
pseudovarieties of that form containing $\pv {Sl}$. The technique used in the proof goes
back to~\cite[Proposition 3.4]{Almeida:1996c}.

For $u\in \Om AV$ consider the following recursive definition:
  \begin{enumerate}
  \item   the left basic factorization of
  $u$ is $u=u_1a_1r_1$;
  \item if $c(r_n)=c(u)$, then $u_{n+1}$, $a_{n+1}$
    and $r_{n+1}$ are such that
    the left basic factorization of
  $r_{n}$ is $r_{n}=u_{n+1}a_{n+1}r_{n+1}$;
  \item if $c(r_n)\subsetneq c(u)$ then the recursive definition stops.
  \end{enumerate}

If this recursive definition stops after a finite number of steps,
that is, if $c(r_n)\subsetneq c(u)$ for some $n\geq 1$,
then $u=u_1a_1u_2a_2\cdots u_{n-1}a_{n-1}u_na_nr_n$
is the~\emph{iterated left basic factorization} of $u$.
Moreover, one says that $u$ has a \emph{finite} iterated left basic
factorization of \emph{length} $\ell=n$ and \emph{remainder} $r=r_n$.
If the recursive definition does not stop
after a finite number of steps then
$u$ has an \emph{infinite iterated left basic
factorization}. We then write $\ell=\infty$.
We also define~$r=1$ if $\ell=\infty$.

The following notation encompasses both cases
$\ell\in\mathbb {N}$ and $\ell=\infty$, where
$\pv {ilbf}(u)$ stands for \emph{iterated left basic
factorization of $u$}:
\begin{equation*}
  \pv {ilbf}(u)=((u_i,a_i)_{1\leq i<\ell+1};r).
\end{equation*}
From the uniqueness of left basic factorizations
one sees
that $\pv {ilbf}(u)$ is well defined.

\begin{remark}\label{r:projecting-lbf}
  Let $\pv V$ be a pseudovariety of semigroups such that
$\pv {Sl}\subseteq \pv V$
and $\pv V=\pv K\malcev \pv V$.
For every $u\in\Om AS$,
\begin{equation*}
  \pv {ilbf}(u)=((u_i,a_i)_{1\leq i<\ell+1};r)
  \implies
  \pv {ilbf}([u]_{\pv V})=(([u_i]_{\pv V},a_i)_{1\leq i<\ell+1};[r]_{\pv V}),
\end{equation*}
where $[1]_{\pv V}=1$.
\end{remark}

\begin{proof}
The result follows from the fact
that $c([w]_{\pv V})=c(w)$ for every $w\in\Om AS$, and
from the uniqueness of iterated left basic factorizations.
\end{proof}

The following proposition expresses
the significance of a pseudoword having infinite
iterated left basic factorization.

\begin{prop}[{\cite[Proposition 3.3]{Trotter&Weil:1997}}]\label{p:ilbf-v-in-ds}
  Let $\pv V$ be a pseudovariety of
  semigroups such that
  $\pv {Sl}\subseteq \pv V\subseteq \pv {DS}$ and $\pv V=\pv K\malcev \pv V$.
    Let $x\in\Om AV$. Then $x$ is regular if and only if $x$ has an
    infinite iterated 
      left basic factorization.
\end{prop}

Let $u$ be an element of $\Om AS$ with length greater than $k$.
Suppose
that
\begin{equation*}
  \pv {ilbf}(\Phi_k(u))=((w_i,z_i)_{1\leq i<\ell+1};r).  
\end{equation*}

For each integer $i$, with $1\le i<\ell+1$, the pseudoword $w_i$ is a factor
of $\Phi_k(u)$ so, by~\cite[Proposition 10.9.2]{Almeida:1994a}, there
is a factor $u_i$ of $u$ such that $\Phi_k(u_i)=w_i$,
and $u_i$ is unique if $w_i\neq 1$ by an application of
Theorem~\ref{t:basic-phi-k} to the pseudovariety $\pv S$.
Also, $\Phi_k(q)=r$ for some $q$, which is unique if
$r\neq 1$.

Suppose that $w_i\neq 1$. Since
$w_iz_i$ and $z_{i-1}w_i$ (the latter being defined only if $i\neq 1$)
are factors of $\Phi_k(u)$, we have $\be k(z_i)=\te k(u_i)$ and
$\te k(z_{i-1})=\be k(u_i)$ (otherwise $\Phi_k(u)$ would have factors
of length $2$ which are not in $\Phi_k(A^+)$).
Therefore, if $k=1$ and $w_i,w_{i+1}\neq 1$,
then $z_i=\te 1(u_i)\,\be 1(u_{i+1})$ when $1\leq i<\ell$.
The case $k=1$ will be sufficient for the proof of
Theorem~\ref{t:more-K-malc-ast-d}
and it is easier to handle with
because in the case $k>1$ we do not have
$z_{i}=\te k(u_i)\be k(u_{i+1})$.

In the following lines we are assuming $k=1$.
If $\ell\in \mathbb N$ and $r=1$, then
we choose $q$ to be the letter $\te 1(u)$.
If some $w_i$ is $1$ then all $w_i$ and $r$
are $1$, and $c(\Phi_1(u))$ has only one element, say $ab$, with
$a,b\in A$.
Then either $u=ab$ or $a=b$.
In the former case, we have $\ell=1$ and $z_1=ab$, and we choose
$u_1$ to be $a$; in the latter case
we have $z_i=aa$ for all $i$, and we choose $u_i$ to be $a$.

So, in every possible case, we have
$z_{i}=\te 1(u_i)\be 1(u_{i+1})$
when $1\leq i<\ell$
and, if $\ell\in\mathbb N$, then $z_{\ell}=\te 1(u_\ell)\be 1(q)$.

If $\ell\in\mathbb {N}$ then
\begin{equation*}
\Phi_1(u_1u_2\cdots u_\ell q)=\Phi_1(u_1)z_1\Phi_1(u_2)z_2\Phi_1(u_3)\cdots
\Phi_1(u_\ell)z_\ell\Phi_1(q)=\Phi_1(u).
\end{equation*}
Since, for $w=u_1u_2\cdots u_\ell q$, we also have $\be 1 (w)=\be 1(u)$
and $\te 1 (w)=\te 1(u)$,
it follows from Theorem~\ref{t:basic-phi-k}, applied to the
pseudovariety $\pv S$, that $u=w$.

More generally, whether $\ell\in\mathbb {N}$ or $\ell=\infty$,
the equality 
\begin{equation*}
\Phi_1(u_1u_2\cdots u_i)=\Phi_1(u_1)z_1\Phi_1(u_2)z_2\Phi_1(u_3)\cdots
\Phi_1(u_{i-1})z_{i-1}\Phi_1(u_{i})
\end{equation*}
holds for all $i$ such that $1\leq i<\ell+1$,
and $u_1u_2\cdots u_i$ is a prefix of $u$.

We define
\begin{equation*}
  \pv {ilfb}_2(u)=((u_i)_{1\leq i<\ell +1};q),
\end{equation*}
where $q=1$ if $\ell=\infty$.

\begin{lema}\label{l:useful-for-the-induction}
    Let $\pv V$ be a pseudovariety of
  semigroups such that $\pv {Sl}\subseteq \pv V$ and
  $\pv V=\pv K\malcev \pv V$. Let $u$ and $v$ be elements of
  $\Om AS\setminus A$ such that
  \begin{equation*}
  \pv {ilfb}_2(u)=((u_i)_{1\leq i<\ell_u +1};q_u),\quad
  \pv {ilfb}_2(v)=((v_i)_{1\leq i<\ell_v +1};q_v).
\end{equation*}
If $\pv V\ast\pv D_1\models u=v$
then $\ell_u=\ell_v$,
$\pv V\ast\pv D_1\models u_i=v_i$ for all $i$ such that
$1\leq i<\ell_u+1$, and
$\pv V\ast\pv D_1\models q_u=q_v$.
\end{lema}

\begin{proof}
  If $\pv V\ast\pv D_1\models u=v$ then
  $\pv V\models \Phi_1(u)=\Phi_1(v)$,
    $\be 1(u)=\be 1(v)$ and $\te 1(u)=\te 1(v)$
    by Theorem~\ref{t:basic-phi-k}.
  By the definition of $\pv {ilfb}_2$, the iterated left basic
  factorizations of $\Phi_1(u)$ and $\Phi_1(v)$ are of the following
  form,
  \begin{align*}
  \pv {ilbf}(\Phi_1(u))&=((\Phi_1(u_i),z_{u,i})_{1\leq i<\ell_u+1};q_u),\\
  \pv {ilbf}(\Phi_1(v))&=((\Phi_1(v_i),z_{v,i})_{1\leq i<\ell_v+1};q_v),
\end{align*}
with $z_{u,i}=\te 1(u_i)\be 1(u_{i+1})$
and $z_{v,i}=\te 1(v_i)\be 1(v_{i+1})$,
where $u_{\ell_u+1}=q_u$
and $v_{\ell_v+1}=q_v$.
  
  Since $\pv V=\pv K\malcev \pv V$,
  the elements of $\Om {A^2}V$
  have unique left basic factorizations. Therefore, taking into
  account  Remark~\ref{r:projecting-lbf}, we conclude that
  $\ell_u$ and $\ell_v$ are equal, say to $\ell$, and
  \begin{align}
    z_{u,i}&=z_{v,i},\label{al:z}\\
    \pv V&\models \Phi_1(u_i)=\Phi_1(v_i),\quad \pv V\models
    \Phi_1(q_u)=\Phi_1(q_v),\label{al:phi} 
  \end{align}
  for all $i$ such that $1\leq i<\ell+1$.

  Since
  $\be 1(u)=\be 1(v)$ and $\te 1(u)=\te 1(v)$,
  we have 
  $\be 1(u_1)=\be 1(v_1)$
  and $\te 1(q_u)=\te 1(q_v)$.
  By~(\ref{al:z})
  we also know that $\be 1(u_{i+1})=\be 1(v_{i+1})$
  and $\te 1(u_i)=\te 1(v_i)$ for every integer $i$
  such that $1\leq i<\ell+1$
  (in particular, $\be 1(q_u)=\be 1(q_v)$).
  Then, from~(\ref{al:phi}) and Theorem~\ref{t:basic-phi-k},
  we obtain $\pv V\ast\pv D_1\models u_i=v_i$
  for every integer $i$
  such that $1\leq i<\ell+1$,
  and $\pv V\ast\pv D_1\models q_u=q_v$.
\end{proof}

\subsection{From $\pv {DS}$ to $\pv {DS}\ast\pv D_k$}\label{sec:from-pv-ds}

In this subsection we deduce some properties of implicit operations
on $\pv {DS}\ast\pv D_k$, for $k\ge 1$.
In fact, we only need the case $k=1$
for proving Theorem~\ref{t:more-K-malc-ast-d}, but the general case
is no more difficult.

The following theorem and corollary state known properties of implicit
operations on $\pv {DS}$ from which we
obtain similar properties
of implicit operations on $\pv {DS}\ast\pv D_k$,
expressed in Theorem~\ref{t:adapt-1-8-i}.

\begin{teor}[{\cite[Theorem
    8.1.7]{Almeida:1994a}}]\label{t:regularity-1-content} 
  Let $u,v\in \Om AS$ be
  such that
  $[u]_{\pv {DS}}$ and $[v]_{\pv {DS}}$ are regular.
  Then
  $[u]_{\pv {DS}}\mathrel{\J}[v]_{\pv {DS}}$
  if and only if
  $c(u)=c(v)$.
\end{teor}

\begin{cor}\label{c:1-8-i}
  Let $u,v\in \Om AS$
  be such that $c(u)=c(v)$.
  Let $\K$ be one of the Green's relations
  $\J$, $\R$ or $\L$. 
  Suppose that
  $[u]_{\pv {DS}}\leq_{\K}[v]_{\pv {DS}}$
  and that $[v]_{\pv {DS}}$ is regular. Then
  $[u]_{\pv {DS}}\mathrel{\K}[v]_{\pv {DS}}$.
\end{cor}

\begin{proof}
  If $[u]_{\pv {DS}}\leq_{\K}[v]_{\pv {DS}}$
  then $[u^\omega]_{\pv {DS}}\leq_{\K}[u]_{\pv {DS}}\leq_{\K}[v]_{\pv
    {DS}}$.
  Since $[u^\omega]_{\pv {DS}}$ is regular and
  $c(u^\omega)=c(u)=c(v)$,
  it follows immediately from Theorem~\ref{t:regularity-1-content}
  that
  $[u]_{\pv {DS}}\mathrel{\J}[v]_{\pv {DS}}$,
  thus $[u]_{\pv {DS}}\mathrel{\K}[v]_{\pv {DS}}$
  by stability of $\Om A{DS}$.
\end{proof}

  \begin{lema}\label{l:regularity-up-and-down}
    Let $u$ be an element of $\Om AS$
    with a factorization of
    the form $u=pxqypzq$, with $p$, $q$ elements of $A^+$ with length
    greater than or equal to~$k$.
    If $[\Phi_k(u)]_{\pv {DS}}$ is regular then
    $[u]_{\pv {DS}\ast\pv D_k}$ is regular.
  \end{lema}

 \begin{proof}
   Clearly, we only need to consider the case where $|p|=|q|=k$.   
   Let $v=(uy)^\omega u$.
      Note that $v\mathrel{\J}(uy)^\omega$, thus $v$ is regular. Therefore,
   it suffices to prove that
   $[v]_{\pv {DS}\ast\pv D_k}=[u]_{\pv {DS}\ast\pv D_k}$.
   
   Consider the idempotents
   $e=\Phi_k(uyp)^\omega$
   and $f=\Phi_k(qyu)^\omega$.
   Note that $e=\Phi_k((uy)^\omega p)$
   and $f=\Phi_k(q(yu)^\omega)$
   because $p=\be k(u)$ and $q=\te k(u)$.
   Moreover, since $v=(uy)^\omega u(yu)^\omega$,
   we have 
   \begin{equation}\label{eq:regularity-up-and-down}
     \Phi_k(v)=e\cdot \Phi_k(u)\cdot f.
   \end{equation}    
   Moreover, the factorization
   $u=pxqypzq$ and the lengths of $p$ and $q$
   assure us that $c_{k+1}(u)=c_{k+1}(uyu)=c_{k+1}(v)$,
   and $c(\Phi_k(v))=c(\Phi_k(u))=c(e)=c(f)$.
   Since $[\Phi_k(u)]_{\pv {DS}}$ is regular,
   by Theorem~\ref{t:regularity-1-content} we know that
   $[\Phi_k(u)]_{\pv {DS}}$,
   $[e]_{\pv {DS}}$ and $[f]_{\pv {DS}}$ belong to the same
   $\J$-class of $\Om {A^{k+1}}{DS}$.
   As $e\leq_{\R}\Phi_k(u)$
   and $f\leq_{\L}\Phi_k(u)$,
   it follows that
   $[e]_{\pv {DS}}\mathrel {\R}[\Phi_k(u)]_{\pv {DS}}$
   and $[f]_{\pv {DS}}\mathrel {\L}[\Phi_k(u)]_{\pv {DS}}$.
   Therefore, from~(\ref{eq:regularity-up-and-down})
   we obtain $[\Phi_k(v)]_{\pv {DS}}=[\Phi_k(u)]_{\pv {DS}}$.
   As $\be k(v)=\be k(u)$ and $\te k(v)=\te k(u)$,
   applying Theorem~\ref{t:basic-phi-k} we
   get $[v]_{\pv {DS}\ast\pv D_k}=[u]_{\pv {DS}\ast\pv D_k}$.   
 \end{proof}

 Since $\pv {DS}=\pv K\malcev\pv {DS}$, the
  elements of $\Om A{DS}$ have unique iterated left basic factorizations, for
  every alphabet~$A$.  

  \begin{teor}\label{t:regularity-up-and-down}
      Let $u$ be an element of $\Om AS$ with length greater than $k$.
      The following conditions are equivalent:
      \begin{enumerate}
      \item the length of the iterated left basic factorization of
      $\Phi_k(u)$ is infinite;\label{item:regularity-up-and-down-1}
      \item $[u]_{\pv {DS}\ast\pv D_k}$ is
      regular;\label{item:regularity-up-and-down-2}
      \item $[\Phi_k(u)]_{\pv {DS}}$ is
        regular.\label{item:regularity-up-and-down-3} 
      \end{enumerate}
 \end{teor}

 \begin{proof}
   (\ref{item:regularity-up-and-down-1})$\Leftrightarrow$(\ref{item:regularity-up-and-down-3}):
   The iterated left basic factorizations of
   $\Phi_k(u)$ and $[\Phi_k(u)]_{\pv {DS}}$
   have the same length (cf.~Remark~\ref{r:projecting-lbf}),
   hence this equivalence follows immediately from
      Proposition~\ref{p:ilbf-v-in-ds}.

   (\ref{item:regularity-up-and-down-2})$\Rightarrow$(\ref{item:regularity-up-and-down-3}):
   This implication follows from Corollary~\ref{c:basic-phi-k}.

   (\ref{item:regularity-up-and-down-1})$\Rightarrow$(\ref{item:regularity-up-and-down-2}):   
   By hypothesis, there is a factorization
   \begin{equation*}
   \Phi_k(u)=w_1z_1w_2z_2w_3z_3w_4z_4w_5,     
   \end{equation*}
   for some
   $w_i\in \Om {A^{k+1}}S$ ($1\le i\le 5$),
   $z_j\in A^{k+1}$ ($1\le j\le 4$),
   such that $c(w_iz_i)=c(\Phi_k(u))$ for $1\le i\le 4$.
   Moreover, we may choose $w_i$ to be an infinite pseudoword,
   for every $i$.

   Let $p=\be {k+1}(u)$ and $q=\te {k+1}(u)$.
   It follows from~\cite[Proposition 10.9.2]{Almeida:1994a}
   that there is some factorization of the form
   $u=pxv_1yv_2zq$, with
   $x,y,z\in(\Om AS)^1$ and $v_1,v_2\in\Om AS$
   such that $\Phi_k(v_1)=w_2z_2$
   and $\Phi_k(v_2)=w_4z_4$
   (the hypothesis that $w_3$ is infinite guarantees that the factors
   $v_1$ and $v_2$ of $u$ appear in a non-overlapping way
   when reading factors of $u$ from left to right).
   Since $q\in c(w_2z_2)=c(\Phi_k(u))$ and $p\in c(w_4z_4)=c(\Phi_k(u))$,
   we conclude that $p$ and $q$ are respectively factors of $v_1$ and $v_2$,
   thus $u=px'qy'pz'q$
   for some $x',y',z'\in (\Om AS)^1$.
   Note that $[\Phi_k(u)]_{\pv {DS}}$ is regular, by the already
   proved equivalence
   (\ref{item:regularity-up-and-down-1})$\Leftrightarrow$(\ref{item:regularity-up-and-down-3}).   
   This shows we are in the
   conditions of Lemma~\ref{l:regularity-up-and-down}, thus
   concluding the proof of the implication (\ref{item:regularity-up-and-down-1})$\Rightarrow$(\ref{item:regularity-up-and-down-2}).
 \end{proof}

 The following result is an analog of Corollary~\ref{c:1-8-i}.
 
\begin{teor}\label{t:adapt-1-8-i}
   Let $u$ and $v$ be elements of $\Om AS$ with length greater than~$k$.
     Let $\K$ be one of Green's relations
  $\J$, $\R$ or $\L$. 
   Suppose that
  $c_{k+1}(u)=c_{k+1}(v)$,
  $[u]_{\pv {DS}\ast\pv D_{k}}\leq_{\K}[v]_{\pv {DS}\ast\pv D_{k}}$
  and that $[v]_{\pv {DS}\ast\pv D_{k}}$ is regular.
  Then $[u]_{\pv {DS}\ast\pv D_{k}}\mathrel {\K}[v]_{\pv {DS}\ast\pv D_{k}}$.
\end{teor}

\begin{proof}
  Since $\Omvar A{(\pv {DS}\ast\pv D_{k})}$ is stable,
  it suffices to consider the case~$\K=\J$.
  
  By Corollary~\ref{c:basic-phi-k},
  we know that
  $[\Phi_k(u)]_{\pv {DS}}\leq_{\J}[\Phi_k(v)]_{\pv {DS}}$
  and that $[\Phi_k(v)]_{\pv {DS}}$ is regular.
  Since $c(\Phi_k(u))=c(\Phi_k(v))$,
  it follows from Corollary~\ref{c:1-8-i} that 
  $[\Phi_k(u)]_{\pv {DS}}\mathrel{\J}[\Phi_k(v)]_{\pv {DS}}$,
  and so $[\Phi_k(u)]_{\pv {DS}}$ is regular.
  Therefore,
  $[u]_{\pv {DS}\ast\pv D_{k}}$ is regular by
  Theorem~\ref{t:regularity-up-and-down}.

  Then there are $x,y\in \Om AS$ such that $[xvy]_{\pv {DS}\ast\pv
    D_k}$
  is a group element in the $\J$-class of
  $[u]_{\pv {DS}\ast\pv D_k}$:
  \begin{equation*}
  \pv {DS}\ast\pv D_k\models u\mathrel {\J}xvy=(xvy)^{\omega+1}.
  \end{equation*}
  
  Let $z=(vyx)^\omega v$. Then
  \begin{equation}\label{eq:adapt-1-8-i-1}
  \Phi_k(z)=\Phi_k((vyx)^\omega\,\be k(v))\cdot \Phi_k(v)
  =\Phi_k(vyx\,\be k(v))^\omega\cdot \Phi_k(v).
  \end{equation}
  Moreover, as $\pv {Sl}\ast\pv D_k\models (xvy)^{\omega+1}\mathrel {\J}u$,
  we have $c_{k+1}((xvy)^{\omega+1})=
  c_{k+1}(u)=c_{k+1}(v)$ and so
    \begin{equation*}
  c(\Phi_k(vyx\,\be
  k(v))^\omega)=c_{k+1}(vyxv)=c_{k+1}((xvy)^{\omega+1})=c(\Phi_k(v)).     
  \end{equation*}
  Since $[\Phi_k(v)]_{\pv {DS}}$ and $[\Phi_k(vyx\,\be k(v))^\omega]_{\pv
    {DS}}$
  are regular elements of $\Om {A^{k+1}}{DS}$,
  and since $\Phi_k(vyx\,\be  k(v))^\omega\leq_{\R}\Phi_k(v)$, it 
  follows
  from Theorem~\ref{t:regularity-1-content}
  that $[\Phi_k(v)]_{\pv {DS}}\mathrel{\R}[\Phi_k(vyx\,\be k(v))^\omega]_{\pv
    {DS}}$.
  
  Therefore, and as $[\Phi_k(vyx\,\be k(v))^\omega]_{\pv {DS}}$
  is idempotent, we deduce
  from~(\ref{eq:adapt-1-8-i-1}) that $\pv {DS}\models \Phi_k(z)=\Phi_k(v)$.
  Clearly, $\be k(z)=\be k(v)$ and $\te k(z)=\te k(v)$. Hence,
  $\pv {DS}\ast\pv D_k\models z=v$. Since
  $\pv {DS}\ast\pv D_k\models z\leq_{\J}xvy\leq_{\J}v$, we
  conclude that $\pv {DS}\ast\pv D_k\models u\mathrel {\J}v$.
\end{proof}

\section{Proof of Theorem~\ref{t:more-K-malc-ast-d}}\label{sec:proof-theor-reft:m}

The following lemma is implicitly shown
in the last paragraph
of the proof of Theorem 3.6 in~\cite{Trotter&Weil:1997}, for the case
of subpseudovarieties of $\pv {DS}$.
We also note that this result fits in
the framework of~\cite[Lemma 4.4.4 and Proposition
4.6.19]{Rhodes&Steinberg:2009}.

\begin{lema}\label{l:argument-used-by-pin-weil}
    Let $\pv V$ be a pseudovariety of semigroups.
    Given an alphabet~$A$, consider the
  canonical homomorphism $\rho:\Om A{(\pv K\malcev\pv V)}\to \Om AV$.
  If $x$ and $y$ are elements of
  $\Om A{(\pv K\malcev\pv V)}$
  such that $y$ is regular, $x\leq_{\R}y$ and $\rho(x)=\rho(y)$, then~$x=y$.
\end{lema}

\begin{proof}
  Suppose first that $y$ is idempotent.
  By the Pin-Weil basis theorem for Mal'cev
  products~\cite{Pin&Weil:1996a}, 
  given an idempotent $e$ of $\Om AV$, the
  semigroup $\rho^{-1}(e)$ is a pro-$\pv K$ semigroup.
  Hence $\rho(x)=\rho(y)$
  implies $yx=y$, and so
  $y\leq_{\L}x$. Since, $x\leq_{\R}y$  
  we deduce that $y$ and $x$
  belong to the same $\H$-class $H$.
  As $y$ is idempotent, $H$ is a group
  with identity $y$. Therefore, we have $y=yx=x$.

  Consider now the general case in which $y$ is regular, but not
  necessarily idempotent. Then, there is $s\in \Om A{(\pv K\malcev\pv V)}$
  such that $y=ysy$.
  Note that $sy$ is idempotent,
  $sx\leq_{\R}sy$ and $\rho(sx)=\rho(sy)$. From the case already
  proved in the previous paragraph, we get $sx=sy$.
  By hypothesis, we may take $t\in \Om A{(\pv K\malcev\pv V)}$
  such that $x=yt$. Then $y=ysy=ysx=ysyt=yt=x$.
\end{proof}

The following auxiliary result was inspired by Theorem 3.6
from~\cite{Trotter&Weil:1997}, with which it has similarities. The main difference is that
the pseudovarieties considered
in Theorem 3.6 from~\cite{Trotter&Weil:1997}
are contained in $\pv {DS}$, while in the following result
we need to go out from that realm and
consider pseudovarieties contained in $\pv {DS}\ast\pv {D_1}$.
The locality of $\pv {DS}$ is crucial in the proof,
as it guarantees pseudovarieties
appearing there are indeed contained in $\pv {DS}\ast\pv {D_1}$.

\begin{prop}\label{p:adaptation-of-result-of-tw}
  Let $\pv V$ be a pseudovariety of semigroups such that
  $\pv {Sl}\subseteq \pv V\subseteq \pv {DS}$.
  Let $u,v\in \Om AS\setminus A$,
  with 
    \begin{equation}
    \label{eq:adaptation-of-result-of-tw-0}
    \pv {ilbf}_2(u)=((u_i)_{1\leq i<\ell_u +1};q_u)
    \quad\text{and}\quad
    \pv {ilbf}_2(v)=((v_i)_{1\leq i<\ell_v +1};q_v).
  \end{equation}  
  Then $\pv K\malcev (\pv V\ast\pv D_1)\models u=v$ if the
  following conditions hold:
   \begin{enumerate}
   \item $\pv V\ast\pv D_1\models u=v$;\label{item:adaptation-of-result-of-tw-1}
   \item $\ell_u=\ell_v$;\label{item:adaptation-of-result-of-tw-2}
   \item $\pv K\malcev (\pv V\ast\pv D_1)\models
     u_i=v_i$ for all $i\geq 1$ and
     $\pv K\malcev (\pv V\ast\pv D_1)\models q_u=q_v$.\label{item:adaptation-of-result-of-tw-4}
   \end{enumerate}
\end{prop}

\begin{proof}
  Let $\ell=\ell_u=\ell_v$.
  If $\ell$ is finite, then
  $u=u_1u_2u_3\cdots u_{\ell}q_u$
  and $v=v_1v_2v_3\cdots v_{\ell}q_v$.
  Therefore, thanks to condition~(\ref{item:adaptation-of-result-of-tw-4}),
  the proof of the theorem for this case is immediate.

  Suppose that $\ell=\infty$.
  Denote by $\pv W$ the pseudovariety
  $\pv K\malcev (\pv V\ast\pv D_1)$.
  Then $\pv W\subseteq \pv K\malcev (\pv {DS}\ast\pv D_1)$.
  The pseudovariety $\pv {DS}$ is local~\cite{Jones&Trotter:1995}
  and one has $\pv K\malcev\pv {DS}=\pv {DS}$.
   Therefore, by Theorem~\ref{t:z-malc-ast-d}, we have
    $\pv K\malcev(\pv {DS}\ast\pv D_1)=\pv {DS}\ast\pv
    D_1$,
  whence $\pv W\subseteq \pv {DS}\ast\pv D_1$.
  Since $\ell=\infty$, the implicit operations
  $[u]_{\pv {DS}\ast\pv D_1}$ and $[v]_{\pv {DS}\ast\pv D_1}$ are
  regular by
  Theorem~\ref{t:regularity-up-and-down}.
  Hence, since $\pv W\subseteq \pv {DS}\ast\pv D_1$,
  the implicit operations $[u]_{\pv W}$ and $[v]_{\pv W}$ are also
  regular.
  
  For each $n\geq 1$,
  let $w_n=u_1u_2\cdots u_n$.
    Note that $u\leq_{\R}w_n$.
  Let $w$ be an accumulation point of the sequence $(w_n)_n$.
  Then $\pv{ilbf}(\Phi_1(w))$ is infinite and
  $u\leq_{\R}w$.
  In particular,
  $[w]_{\pv {DS}\ast\pv D_1}$ is regular by
  Theorem~\ref{t:regularity-up-and-down}.
  Moreover, $c_2(w)=c_2(u)$, and so applying
  Theorem~\ref{t:adapt-1-8-i} we obtain
  $[u]_{\pv {DS}\ast\pv D_1}\mathrel{\R}[w]_{\pv {DS}\ast\pv D_1}$.
    Since $\pv W\subseteq \pv {DS}\ast\pv D_1$,
  we conclude that
  $[u]_{\pv W}\mathrel{\R}[w]_{\pv W}$.  
  On the other hand, since $v_1\cdots v_n$ is a prefix of
  $v$,
  by Condition~(\ref{item:adaptation-of-result-of-tw-4}) in the
  statement,
  we have $[v]_{\pv W}\leq_{\R}[w_n]_{\pv W}$
  for all $n$.
  Hence $[v]_{\pv W}\leq_{\R}[u]_{\pv W}$.
  Since $[u]_{\pv V\ast\pv D_1}=[v]_{\pv V\ast\pv D_1}$,
  applying Lemma~\ref{l:argument-used-by-pin-weil} we get
  $[u]_{\pv W}=[v]_{\pv W}$.  
\end{proof}

\begin{prop}\label{p:basis-method}
  Let $\pv V$ be a pseudovariety of semigroups such that
  $\pv {Sl}\subseteq \pv V\subseteq \pv {DS}$.
  Then $(\pv K\malcev \pv V)\ast\pv D_1\models u=v$ implies
  $\pv K\malcev (\pv V\ast\pv D_1)\models u=v$,
  for every $u,v\in\Om AS$.
\end{prop}

\begin{proof}  
  Since $(\pv K\malcev \pv V)\ast\pv D_1\models u=v$,
  in particular we have 
  $\pv K\models u=v$. Therefore, $u\in A^+$ if and only if $v\in A^+$,
  and if $u,v\in A^+$ then $u=v$.
   
  Next, we prove the proposition by induction on $|c_2(u)|$.
  If $|c_2(u)|=0$, then $u\in A$,
  and so $u=v$.

  Suppose that $|c_2(u)|>0$, and that the proposition holds for
  pseudoidentities between elements with fewer
  factors of length two than $u$ has.
  Let
  \begin{equation*}
  \pv {ilfb}_2(u)=((u_i)_{1\leq i<\ell_u +1};q_u),\quad
  \pv {ilfb}_2(v)=((v_i)_{1\leq i<\ell_v +1};q_v).
\end{equation*}
  Since $(\pv K\malcev \pv V)\ast\pv D_1\models u=v$,
  from Lemma~\ref{l:useful-for-the-induction}
  we deduce that $\ell_u$ and $\ell_v$ are equal, say to
  $\ell$,
  and that
    \begin{equation}
    \label{eq:basis-method-3}
    (\pv K\malcev \pv V)\ast\pv D_1\models
    u_i=v_i\quad\text{and}\quad
    (\pv K\malcev \pv V)\ast\pv D_1\models q_u=q_v,
  \end{equation}
  for all $i$ such that $1\leq i<\ell+1$.
  
  By the definition of $\pv {ilfb}_2$, we know that $c_2(u_i)$ and
  $c_2(q_u)$ have fewer elements than $c_2(u)$. Applying the induction
  hypothesis to (\ref{eq:basis-method-3}), we obtain
      \begin{equation}
    \label{eq:basis-method-4}
    \pv K\malcev (\pv V\ast\pv D_1)\models
    u_i=v_i\quad\text{and}\quad
    \pv K\malcev (\pv V\ast\pv D_1)\models q_u=q_v,
  \end{equation}
  for all $i$ such that $1\leq i<\ell+1$.  
  Therefore $\pv K\malcev (\pv V\ast\pv D_1)\models u=v$
  by Proposition~\ref{p:adaptation-of-result-of-tw}.
\end{proof}

 We are now ready to conclude the proof of
 Theorem~\ref{t:more-K-malc-ast-d}.
 We do it in three steps, one for each of the
  pseudovarieties $\pv K$, $\pv D$ and $\Lo I$.

 \begin{proof}[The case of pseudovariety $\pv K$]
   We want to show that the following inclusion holds for every
   pseudovariety of semigroups $\pv V$ containing $\pv {Sl}$, and
   every $k\geq 1$:
   \begin{equation}\label{eq:first-step}
   \pv K\malcev (\pv V\ast\pv D_k)\subseteq
   (\pv K\malcev \pv V)\ast\pv D_k.        
   \end{equation}  
  If $B_2\in\pv V$, then there is nothing to prove thanks
to Theorem~\ref{t:more-z-malc-ast-d}.
Suppose that $B_2\notin\pv V$, that is,
$\pv V\subseteq \pv {DS}$.

  Since every pseudovariety is defined by a basis of pseudoidentities,
  Proposition~\ref{p:basis-method} states that
  \begin{equation}\label{eq:special-case-k-is-1}
      \pv K\malcev (\pv V\ast\pv D_1)\subseteq (\pv K\malcev \pv V)\ast\pv D_1.
  \end{equation}

  It is well known that $B_2\in \pv V\ast\pv D_1$
  (cf.~\cite[Theorem 5.7]{Straubing:1985}, for example).
  Let $k$ be an integer greater than $1$. Since
  $\pv V\ast\pv D_k=\pv V\ast\pv D_1\ast\pv D_{k-1}$
  (recall that $\pv {D}_m\ast\pv {D}_n=\pv {D}_{m+n}$
  for all $m,n\geq 1$~\cite[Lemma 10.4.1]{Almeida:1994a})
  and $B_2\in \pv V\ast\pv D_1$, we may apply
  Theorem~\ref{t:more-z-malc-ast-d} to get
  \begin{equation*}
    \pv K\malcev (\pv V\ast\pv D_k)
    \subseteq
    (\pv K\malcev (\pv V\ast\pv D_1))\ast\pv D_{k-1}.
  \end{equation*}
  By~(\ref{eq:special-case-k-is-1}), we then obtain
  \begin{equation*}
    \pv K\malcev (\pv V\ast\pv D_k)\subseteq
    (\pv K\malcev \pv V)\ast\pv D_1\ast\pv D_{k-1},
  \end{equation*}
  which is precisely inclusion~(\ref{eq:first-step}).
\end{proof}

\begin{proof}[The case of pseudovariety $\pv D$]
  The inclusion
   \begin{equation}\label{eq:second-step}
   \pv D\malcev (\pv V\ast\pv D_k)\subseteq
   (\pv D\malcev \pv V)\ast\pv D_k.
   \end{equation}
   is obtained by applying
   Propositions~\ref{p:reverse-VastD} and~\ref{p:reverse-malcev}
   to the inclusion~(\ref{eq:first-step}).   
\end{proof}

\begin{proof}[The case of pseudovariety $\Lo I$]
   Let $\alpha$ and $\beta$ be the operators on the lattice of
   semigroup pseudovarieties defined by
   $\alpha(\pv W)=\pv K\malcev \pv W$ and $\beta(\pv W)=\pv D\malcev \pv W$,
   where $\pv W$ is a pseudovariety of semigroups.
   It is proved in~\cite{Auinger&Hall&Reilly&Zhang:1995}
   (more precisely, see Theorems 6.4 and 8.2
   from~\cite{Auinger&Hall&Reilly&Zhang:1995})
   that
   \begin{equation}\label{eq:Li-hierarchy}
   \Lo I\malcev \pv W
   =\bigcup_{n\ge 0}(\beta\circ\alpha)^n(\pv W)
   =\bigcup_{n\ge 0}(\alpha\circ\beta)^n(\pv W)
   \end{equation}
   for every pseudovariety of semigroups
   $\pv W$.
   If $\pv V$ is a pseudovariety of semigroups containing
   $\pv {Sl}$, then
   it follows easily from~(\ref{eq:first-step})
   and~(\ref{eq:second-step}), and from induction on $n$,
   that $(\beta\circ\alpha)^n(\pv V\ast\pv D_k)\subseteq
   (\beta\circ\alpha)^n(\pv V)\ast\pv D_k$
   for every $n\ge 0$ and~$k\ge 1$.
   Thanks to (\ref{eq:Li-hierarchy}),
   this shows $\Lo I\malcev (\pv V\ast\pv D_k)\subseteq
   (\Lo I\malcev \pv V)\ast\pv D_k$.
\end{proof}

All cases were considered, whence
Theorem~\ref{t:more-K-malc-ast-d} is proved.\qed

\appendix

\section{Proof of Proposition~\ref{p:local-operator-comutes}}

To write an explicit proof of
Proposition~\ref{p:local-operator-comutes}
(which, as remarked in the introduction, is a straightforward
generalization of the result presented
in~\cite[Exercise 4.6.58]{Rhodes&Steinberg:2009})
we
need the language and results of the
``Semilocal Theory'' of finite semigroups, the subject
of Section 4.6 in~\cite{Rhodes&Steinberg:2009}. We follow some
definitions and notations
from~\cite[Subsections 4.6.1 and 4.6.2]{Rhodes&Steinberg:2009}.
However, we also introduce definitions and notations
of our own, in order to make a
uniform proof of Proposition~\ref{p:local-operator-comutes}.

We use
the expression
\emph{$\pv K$-semigroup}
for a right mapping semigroup
and the expression
\emph{$(\pv K\vee \pv G)$-semigroup}
for a right letter mapping semigroup.
The expressions \emph{$\pv D$-semigroup}
and \emph{$(\pv D\vee \pv G)$-semigroup}
have dual meanings.
Finally, the expression
\emph{$\Lo I$-semigroup}
is used for
generalized group mapping semigroups,
and
\emph{$\Lo G$-semigroup} for
generalized group mapping semigroups with an aperiodic distinguished
$\J$-class.

The following result is also from~\cite{Rhodes&Steinberg:2009} (cf.~\cite[Proposition 4.6.56 and the comment following its proof]{Rhodes&Steinberg:2009}).

\begin{prop}\label{p:local-submonoids-are-also-xmapping} 
  Let $\pv Z\in\mathbb V\setminus\{\pv N,\pv N\vee\pv G\}$.
  If $S$ is a $\pv Z$-semigroup
  then every local monoid of $S$ is a $\pv Z$-semigroup.
\end{prop}

Next, we uniformize some of the definitions appearing
in~\cite[Subsection 4.6.2]{Rhodes&Steinberg:2009}.
Let $S$ be a finite semigroup
and let $J$ be a regular $\J$-class of $S$.
The canonical homomorphisms 
\begin{center}
$S\to \pv {AGGM}_J(S)$,
$S\to \pv {GGM}_J(S)$,
$S\to \pv {RLM}_J(S)$
and 
$S\to \pv {RM}_J(S)$  
\end{center}
will be denoted respectively by
$\mu_{\Lo G,J}^S$,
$\mu_{\Lo I,J}^S$,
$\mu_{\pv K\vee\pv G,J}^S$
and $\mu_{\pv K,J}^S$.

Similarly, the canonical homomorphisms
\begin{center}
$S\to S/{\pv {AGGM}}$,
$S\to S/{\pv {GGM}}$,
$S\to S/{\pv {RLM}}$,
and
$S\to S/{\pv {RM}}$  
\end{center}
will be denoted respectively by
$\mu_{\Lo G}^S$,
$\mu_{\Lo I}^S$,
$\mu_{\pv K\vee\pv G}^S$
and $\mu_{\pv K}^S$.

The homomorphisms
$\mu_{\pv D\vee\pv G,J}^S$, $\mu_{\pv D,J}^S$,
$\mu_{\pv D\vee\pv G}^S$ and $\mu_{\pv D}^S$
have the obvious dual definitions.
Let $Jr(S)$ be the set of regular $\J$-classes of $S$.
Recall from~\cite[Definition 4.6.44]{Rhodes&Steinberg:2009}
that $\Ker {\mu_{\pv Z}^S}=\bigcap_{J\in Jr(S)}\Ker {\mu_{\pv Z,J}^S}$
for every $\pv Z\in \mathbb V\setminus\{\pv N\vee\pv G,\pv N\}$,
whence $\mu_{\pv Z}^S(S)$ is a subdirect product
of $\prod_{J\in Jr(S)}\mu_{\pv Z,J}^S(S)$.

We divide the proof of
Proposition~\ref{p:local-operator-comutes} in two parts.

\begin{proof}[Proof of Proposition~\ref{p:local-operator-comutes}]
  We treat first the case $\pv Z\in\mathbb V\setminus\{\pv N,\pv
  N\vee\pv G\}$.
  
  Suppose that $S\in \pv Z\malcev \Lo {\pv V}$. Then
  $\mu_{\pv Z}^S(S)\in \Lo V$,
  by~\cite[Theorem 4.6.50]{Rhodes&Steinberg:2009}.
  If $e$ is an idempotent of $S$, then
  $\mu_{\pv Z}^S(eSe)$ is a local monoid of
  $\mu_{\pv Z}^S(S)$, thus $\mu_{\pv Z}^S(eSe)\in \pv V$.
  When $f$ is an idempotent of
  $\mu_{\pv Z}^S(S)$, the
  subsemigroup $(\mu_{\pv Z}^S)^{-1}(f)$ of $S$
  belongs to $\pv Z$,
  by~Theorem 4.6.46 from~\cite{Rhodes&Steinberg:2009},
  and (depending on $\pv Z$) by
  Propositions 4.6.11, 4.6.13 or 4.6.19
  from~\cite{Rhodes&Steinberg:2009}.
    Therefore $eSe\in \pv Z\malcev \pv V$.
  This proves $\pv Z\malcev \Lo V\subseteq \Lo {(\pv Z\malcev \pv V)}$.

  Conversely, suppose that $S\in \Lo {(\pv Z\malcev \pv V)}$.
  Consider a regular $\J$-class $K$ of $S$.
    Let $e$ be an idempotent of $S$.
  Then $M=\mu_{\pv Z,K}^S(eSe)$ is a local monoid of
  $\mu_{\pv Z,K}^S(S)$, and every local monoid of
  $\mu_{\pv Z,K}^S(S)$ is of this form.  
  The semigroup $\mu_{\pv Z,K}^S(S)$ is a $\pv Z$-semigroup by
  Proposition 4.6.29,
  4.6.31, 4.6.35
  or 4.6.37 from \cite{Rhodes&Steinberg:2009},
  depending on which $\pv Z$ is considered.
  Hence $M$ is also a
  $\pv Z$-semigroup by
  Proposition~\ref{p:local-submonoids-are-also-xmapping}.
  If 
  $J$ is the distinguished $\J$-class of $M$ then,
  using the characterizations of
  $\pv Z$-semigroup
  and $\mu_{\pv Z,J}^M$,
  we conclude that $\mu_{\pv Z,J}^M$
  is the identity,
  and so $\mu_{\pv Z}^M$ is also the identity.
  Since $S\in \Lo {(\pv Z\malcev \pv V)}$
  and the monoid $M$ is a homomorphic image of
  $eSe$, we know that $M\in \pv Z\malcev \pv V$.
  And since $\mu_{\pv Z}^M(M)=M$,
  we actually have $M\in\pv V$ by~\cite[Theorem 4.6.50]{Rhodes&Steinberg:2009}.
  This proves that
  $\mu_{\pv Z,K}^S(S)\in \Lo V$.
  As $\mu_{\pv Z}^S(S)$ divides $\prod_{K\in Jr(S)}\mu_{\pv
    Z,K}^S(S)$,
  we deduce that $\mu_{\pv Z}^S(S)\in \Lo V$,
  and so $S\in\pv Z\malcev \Lo V$,
  again by \cite[Theorem 4.6.50]{Rhodes&Steinberg:2009}.
  This concludes the proof of the inclusion
  $\Lo {(\pv Z\malcev \pv V)}\subseteq\pv Z\malcev \Lo V$
  in the case $\pv Z\in\mathbb V\setminus\{\pv N,\pv
  N\vee\pv G\}$.

  Finally, let us prove the proposition for
  $\pv Z\in\{\pv N,\pv N\vee\pv G\}$.  
  In this case $\pv Z=\pv Z_1\cap \pv Z_2$
  for some $\pv Z_1,\pv Z_2\in \mathbb V$.
  Then
  by~\cite[Corollary 3.2]{Pin&Weil:1996a}
  and the already proved cases, we have
  \begin{equation*}
    \pv Z\malcev \Lo V
    =\pv Z_1\malcev \Lo V\cap \pv Z_2\malcev \Lo V
    =\Lo {(\pv Z_1\malcev \pv V)}\cap \Lo {(\pv Z_2\malcev \pv V)}.    
  \end{equation*}
  Since $\Lo {(\pv W_1\cap \pv W_2)}=\Lo  W_1\cap \Lo W_2$ for every
  pair of pseudovarieties $\pv W_1$, $\pv W_2$,
  and again taking into account~\cite[Corollary 3.2]{Pin&Weil:1996a},
  we therefore obtain  $\pv Z\malcev \Lo V=\Lo {(\pv Z\malcev \pv V)}$.
\end{proof}

\providecommand{\bysame}{\leavevmode\hbox to3em{\hrulefill}\thinspace}
\providecommand{\MR}{\relax\ifhmode\unskip\space\fi MR }
\providecommand{\MRhref}[2]{%
  \href{http://www.ams.org/mathscinet-getitem?mr=#1}{#2}
}
\providecommand{\href}[2]{#2}

\end{document}